\documentclass[12pt]{article}
\usepackage[letterpaper, left=2.5cm, right=2.5cm, top=1.5cm, bottom=1.5cm, headheight=7mm, headsep=0.8cm, footskip=1.5cm, includeheadfoot]{geometry}
\usepackage[utf8]{inputenc}
\usepackage{amsmath}
\usepackage{authblk}
\usepackage{amssymb}
\usepackage{bm}
\usepackage{amsthm}
\usepackage{mathtools}
\usepackage{algorithm}
\usepackage{algpseudocode}
\usepackage{color}
\usepackage{graphicx}
\usepackage{subcaption}
\usepackage{MnSymbol}
\usepackage{hyperref}
\hypersetup{
    unicode=false,  
    pdftoolbar=true,
    pdfmenubar=true,
    pdffitwindow=false,
    pdfstartview={FitH},
    pdfauthor={Eero Hirvijoki},
    pdfcreator={Eero Hirvijoki},
    pdfnewwindow=true,
    colorlinks=true,
    linkcolor=blue,
    citecolor=blue,
    filecolor=blue,
    urlcolor=blue
}

\newtheorem{theorem}{Theorem}
\newtheorem{proposition}{Proposition}
\newtheorem{lemma}{Lemma}
\newtheorem{definition}{Definition}
\newtheorem{remark}{Remark}
\newtheorem{corollary}{Corollary}

\theoremstyle{definition}
\newtheorem{example}{Example}

\newcommand{\ehat}{\bm{e}}
\newcommand{\bhat}{\bm{b}}

\title{Normal stability of slow manifolds in nearly-periodic Hamiltonian systems}
\author[1]{J. W. Burby}
\author[2]{E. Hirvijoki}
\affil[1]{Los Alamos National Laboratory, Los Alamos, NM 87545, USA}
\affil[2]{Department of Applied Physics, Aalto University, P. O. Box 11100, 00076 AALTO, Finland  }
\date{March 2021}

\begin{document}

\maketitle

\begin{abstract}
M. Kruskal showed that each nearly-periodic dynamical system admits a formal $U(1)$ symmetry, generated by the so-called roto-rate. We prove that such systems also admit nearly-invariant manifolds of each order, near which rapid oscillations are suppressed. We study the nonlinear normal stability of these slow manifolds for nearly-periodic Hamiltonian systems on \emph{barely symplectic manifolds} -- manifolds equipped with closed, non-degenerate $2$-forms that may be degenerate to leading order. In particular, we establish a sufficient condition for long-term normal stability based on second derivatives of the well-known adiabatic invariant. We use these results to investigate the problem of embedding guiding center dynamics of a magnetized charged particle as a slow manifold in a nearly-periodic system. We prove that one previous embedding, and two new embeddings enjoy long-term normal stability, and thereby strengthen the theoretical justification for these models.
\end{abstract}

\tableofcontents

\section{Introduction}


Dynamical systems with multiple timescales often exhibit special slow trajectories that lie along almost invariant sets known as slow manifolds. \cite{Gorban_2004,MacKay_2004,Burby_Klotz_2020} Because slow manifolds generally fail to be true invariant objects, analysis of their dynamical significance requires special care. The most well-understood case is normal hyperbolicity: the slow manifold attracts or repels nearby trajectories. As shown by Fenichel \cite{Fenichel_1979}, given a normally-hyperbolic slow manifold, nearby there must be a true normally-hyperbolic invariant manifold. Less well-undestood are the normally-elliptic slow manifolds: slow manifolds around which nearby trajectories oscillate. These objects may fail to approximate true invariant manifolds, but they frequently form the basis for model reduction in dynamical systems with weak dissipation. For instance, the equations governing quasigeostrophic flow describe motion on an elliptic slow manifold inside of the rotating shallow water model \cite{Lorenz_1986,Vautard_1986,Vanneste_2004}; normal oscillations correspond to fast gravity waves. Other examples include the incompressible Euler equations \cite{Klainerman_1982}, which arise as a slow manifold inside of the compressible Euler equations, and ideal magnetohydrodynamics, which may be understood as a slow manifold for a pair of charged ideal fluids \cite{Burby_two_fluid_2017}. 


Reduction to an elliptic slow manifold is fraught with theoretical challenges. Principal among these is the question of normal stability: Do trajectories that begin near an elliptic slow manifold exhibit secular normal drifts? While ellipticity implies marginal linear stability on short timescales, normal instability may still arise at later times due to resonance phenomenona. Thus a dynamical model obtained by reduction to an elliptic slow manifold may spontaneously break down. Even worse, such breakdown may be undetectable from within the reduced model itself.

An important example of an elliptic slow manifold for which normal stability remains an open problem was constructed recently by Xiao and Qin \cite{Xiao_2020}, who proposed a novel method for symplectic integration of the so-called guiding center equations for charged particles in a strong magnetic field. As shown by Littlejohn in \cite{Littlejohn_1981,Littlejohn_1982,Littlejohn_1983,Littlejohn_1984}, the guiding center equations comprise a Hamiltonian system with a non-canonical symplectic structure. Moreover, the natural Lagrangian for guiding centers is degenerate; its velocity Hessian is singular. This makes the formulation of symplectic integrators for guiding center dynamics extremely challenging because standard symplectic integration theory is intended for either canonical symplectic structures, or non-degenerate Lagrangians. However, Xiao and Qin suggested a method for circumventing this difficulty: embed the guiding center system as an elliptic slow manifold in a larger system that does admit a regular Lagrangian. Applying conventional symplectic integration methods to this larger system then leads to a higher-dimension structure-preserving scheme with a slow manifold that formally recovers the guiding center dynamics of interest. (Their scheme is therefore a slow manifold integrator, as described in \cite{Burby_Klotz_2020}.) This idea provides an elegant solution to symplectic integration of guiding center dynamics, provided the integrator's slow manifold is normally-stable. On the other hand, if a normal instability does exist, then the scheme will fail after the instability onset time. At present, normal stability remains an open question.

MacKay suggests in \cite{MacKay_2004} that a useful method for establishing elliptic normal stability in general is identifying an adiabatic invariant whose set of critical points gives the slow manifold. Then sign-definiteness of the normal Hessian should imply normal stability by a Lyapunov-type argument. We say that the slow manifold satisfies a \emph{free-action principle}. However, it is unclear in general how to identify such an adiabatic invariant, or even if such an adiabatic invariant exists.

In this Article, we will identify an important class of elliptic slow manifolds for which the free-action principle always applies. In particular we will show that each nearly-periodic system with an appropriate Hamiltonian structure admits elliptic slow manifolds of arbitrary order, and that these slow manifolds coincide with critical sets for adiabatic invariants. Moreover, we will prove rigorously that sign-definiteness of the normal Hessian implies long-term normal stability. After establishing this theoretical result, we will apply it in the study of guiding center dynamics of individual charged particles in a strong magnetic field. 


Using our general theory, we will show that Xiao and Qin's slow manifold embedding of guiding center dynamics enjoys normal stability in continuous time. This result leaves normal stability in discrete time an open question, but motivates further study in that direction. We will also construct a pair of alternative finite-dimensional slow manifold embeddings of guiding center dynamics. One is a covariant relativistic generalization of the Xiao-Qin embedding. The other is a special case of a more general embedding that applies to \emph{any} symplectic Hamiltonian system. Like Xiao and Qin's case, the larger systems into which we embed come equipped with a regular Lagrangian structure. We use our general theory to prove long-term normal stability for these new embeddings, and thereby identify promising future extensions of Xiao and Qin's idea.

In order to ensure that our abstract theory is general enough to handle the guiding center system, we were forced to consider Hamiltonian systems on symplectic manifolds whose symplectic forms may be very nearly degenerate. We formalize this near-degeneracy by supposing the symplectic form $\Omega_\epsilon$ is a smooth function of the parameter $\epsilon$ that quantifies the timescale separation in a nearly-periodic system, and that $\Omega_0$ may be degenerate. We call manifolds equipped with symplectic forms of this type \emph{barely-symplectic}. By working at this level of generality, our abstract results exhibit an interesting competition between stabilizing and destabilizing influences on the slow manifolds that we construct. Namely, stronger degeneracy of $\Omega_\epsilon$ as $\epsilon\rightarrow 0$ appears to destabilize the slow manifolds, while vanishing of early terms in the adiabatic invariant series (see \cite{Burby_Squire_2020} for explicit formulas for the first few terms) has a stabilizing effect. From this perspective, the guiding center  embeddings we study are remarkable because the stabilizing and destabilizing influences balance, leading to normal stability results that would be expected for $\epsilon$-independent symplectic manifolds.

\subsection{Notational conventions}
In this article, smooth shall always mean $C^\infty$. We reserve the symbol $M$ for a smooth manifold equipped with a smooth auxilliary Riemannian metric $g$. We say $f_\epsilon:M_1\rightarrow M_2$, $\epsilon\in\mathbb{R}$, is a smooth $\epsilon$-dependent mapping between manifolds $M_1,M_2$ when the mapping $M_1\times\mathbb{R}\rightarrow M_2:(m,\epsilon)\mapsto f_\epsilon(m)$ is smooth. Similarly, $\bm{T}_\epsilon$ is a smooth $\epsilon$-dependent tensor field on $M$ when (a) $\bm{T}_\epsilon(m)$ is an element of the tensor algebra $\mathcal{T}_m(M)$ at $m$ for each $m\in M$  and $\epsilon\in\mathbb{R}$, and (b) $\bm{T}_\epsilon$ is a smooth $\epsilon$-dependent mapping between the manifolds $M$ and $\mathcal{T}(M)=\cup_{m\in M}\mathcal{T}_m(M)$. 

The symbol $X_\epsilon$ will always denote a smooth $\epsilon$-dependent vector field on $M$. If $\bm{T}_\epsilon$ is a smooth $\epsilon$-dependent section of either $TM\otimes TM$ or $T^*M\otimes T^*M$ then $\widehat{\bm{T}}_\epsilon$ is the corresponding smooth $\epsilon$-dependent bundle map $T^*M\rightarrow TM:\alpha\mapsto \iota_\alpha\bm{T}_\epsilon$, or $TM\rightarrow T^*M:X\mapsto \iota_X\bm{T}_\epsilon$, respectively. Note that if $\Omega$ is a symplectic form on $M$ with associated Poisson bivector $\mathcal{J}$ then $\widehat{\Omega}^{-1} = -\widehat{\mathcal{J}}$.

\section{Kruskal's theory of nearly-periodic systems\label{nps_background}}

In 1962, Kruskal presented an asymptotic theory \cite{Kruskal_1962} of averaging for dynamical systems whose trajectories are all periodic to leading order. Nowadays, Kruskal's method is termed one-phase averaging \cite{Lochak_1993}, which suggests a contrast with the multi-phase averaging methods underlying, e.g. Kolmogorov-Arnol'd-Moser (KAM) theory. Since this theory provides the framework for the results in this Article, we review its main ingredients here.

\begin{definition}
A \textbf{nearly-periodic system} on a manifold $M$ is a smooth $\epsilon$-dependent vector field $X_\epsilon$ on $M$ such that $X_0 = \omega_0\,R_0$, where
\begin{itemize}
    \item $\omega_0:M\rightarrow\mathbb{R}$ is strictly positive
    \item $R_0$ is the infinitesimal generator for a circle action $\Phi^0_\theta:M\rightarrow M$, $\theta\in U(1)$.
    \item $\mathcal{L}_{R_0}\omega_0 = 0$.
\end{itemize}
The vector field $R_0$ is called the \textbf{limiting roto-rate}, and the set $S_0 = \{s\in M\mid R_0(s) = 0\}$ is called the \textbf{limiting slow manifold}.
\end{definition}
\begin{remark}
In addition to requiring $\omega_0$ is sign-definite, Kruskal assumed that $R_0$ is nowhere vanishing. However, this assumption is not essential for one-phase averaging to work. In fact, the limiting slow manifold $S_0$ will play an crucial role in the rest of this Article. Note that, by Lemma \ref{S0_is_submanifold}, $S_0\subset M$ is indeed a submanifold.
\end{remark}

Kruskal's theory applies to both Hamiltonian and non-Hamiltonian systems. In the Hamiltonian setting, it leads to stronger conclusions. A general class of Hamiltonian systems for which the theory works nicely may be defined as follows.

\begin{definition}
Let $(M,\Omega_\epsilon)$ be a manifold equipped with a smooth $\epsilon$-dependent presymplectic form $\Omega_\epsilon$. Assume there is a smooth $\epsilon$-dependent $1$-form $\vartheta_\epsilon$ such that $\Omega_\epsilon = - \mathbf{d}\vartheta_\epsilon$. A \textbf{nearly-periodic Hamiltonian system} on $(M,\Omega_\epsilon)$ is a nearly-periodic system $X_\epsilon$ on $M$ such that $\iota_{{X}_\epsilon}\Omega_\epsilon = \mathbf{d}H_\epsilon$, for some smooth $\epsilon$-dependent function $H_\epsilon:M\rightarrow\mathbb{R}$. 
\end{definition}

Kruskal showed that all nearly-periodic systems admit an approximate $U(1)$-symmetry that is determined to leading order by the unperturbed periodic dynamics. He named the generator of this approximate symmetry the \emph{roto-rate}. In the Hamiltonian setting, he showed that both the dynamics and the Hamiltonian structure are $U(1)$-invariant.

\begin{definition}
A \textbf{roto-rate} for a nearly-periodic system $X_\epsilon$ on a manifold $M$ is a formal power series $R_\epsilon = R_0 + \epsilon\,R_1 + \epsilon^2\,R_2 + \dots$ with vector field coefficients such that
\begin{itemize}
    \item $R_0$ is equal to the limiting roto-rate
    \item $\exp(2\pi \mathcal{L}_{R_\epsilon}) = 1$
    \item $[X_\epsilon,R_\epsilon] = 0$,
\end{itemize}
where the second and third conditions are understood in the sense of formal power series.
\end{definition}

\begin{proposition}[Kruskal \cite{Kruskal_1962}]\label{existence_of_roto_rate}
Every nearly-periodic system admits a unique roto-rate $R_\epsilon$. The roto-rate for a nearly-periodic Hamiltonian system on an exact presymplectic manifold $(M,\Omega_\epsilon)$ satisfies $L_{R_\epsilon}\Omega_\epsilon = 0$ in the sense of formal power series. 
\end{proposition}

\begin{corollary}\label{invariance_of_Hamiltonian}
The roto-rate $R_\epsilon$ for a nearly-periodic Hamiltonian system $X_\epsilon$ on an exact presymplectic manifold $(M,\Omega_\epsilon)$ with Hamiltonian $H_\epsilon$ satisfies $L_{R_\epsilon}H_\epsilon = 0$.
\end{corollary}
\begin{proof}
Since $[R_\epsilon,X_\epsilon] = \mathcal{L}_{R_\epsilon}X_\epsilon = 0$ and $\mathcal{L}_{R_\epsilon}\Omega_\epsilon = 0$, we may apply the Lie derivative $\mathcal{L}_{R_\epsilon}$ to Hamilton's equation $\iota_{X_\epsilon}\Omega_\epsilon = \mathbf{d}H_\epsilon$ to obtain
\begin{align*}
\mathcal{L}_{R_\epsilon}(\mathbf{d}H_\epsilon) = \mathcal{L}_{R_\epsilon}(\iota_{X_\epsilon}\Omega_\epsilon) = \iota_{\mathcal{L}_{R_\epsilon} X_\epsilon}\Omega_\epsilon + \iota_{X_\epsilon}(\mathcal{L}_{R_\epsilon}\Omega_\epsilon) = 0.
\end{align*}
Thus, $\mathcal{L}_{R_\epsilon}H_\epsilon$ is a constant function. But by averaging over the $U(1)$-action, we conclude that the constant must be zero.
\end{proof}

To prove Proposition \ref{existence_of_roto_rate}, Kruskal used a pair of technical results, each of which is interesting in its own right. The first establishes the existence of a non-unique normalizing transformation that asymptotically deforms the $U(1)$ action generated by $R_\epsilon$ into the simpler $U(1)$-action generated by $R_0$. The second is a subtle bootstrapping argument that upgrades leading-order $U(1)$-invariance to all-orders $U(1)$-invariance for integral invariants. We state these results here for future reference.

\begin{definition}
Let $G_\epsilon =\epsilon \,G_1+ \epsilon^2\,G_2 + \dots $ be an $O(\epsilon)$ formal power series whose coefficients are vector fields on a manifold $M$. 
The \textbf{Lie series} with \textbf{generator} $G_\epsilon$ is the formal power series $\exp(\mathcal{L}_{G_\epsilon})$ whose coefficients are differential operators on the tensor algebra over $M$. 
\end{definition}

\begin{definition}\label{normalizing_transformation_def}
A \textbf{normalizing transformation} for a nearly-periodic system $X_\epsilon$ with roto-rate $R_\epsilon$ is a Lie series $\exp(\mathcal{L}_{G_\epsilon})$ with generator $G_\epsilon$ such that $R_\epsilon = \exp(\mathcal{L}_{G_\epsilon})R_0$.
\end{definition}

\begin{proposition}[Kruskal]\label{existence_of_normalizing_transformation}
Each nearly-periodic system admits a normalizing transformation.
\end{proposition}

\begin{proposition}\label{bootstrap_prop}
Let $\alpha_\epsilon$ be a smooth $\epsilon$-dependent differential form on a manifold $M$. Suppose $\alpha_\epsilon$ is an absolute integral invariant for a $C^\infty$ nearly-periodic system $X_\epsilon$ on $M$. If $\mathcal{L}_{R_0}\alpha_0 = 0$ then $\mathcal{L}_{R_\epsilon}\alpha_\epsilon = 0$, where $R_\epsilon$ is the roto-rate for $X_\epsilon$.
\end{proposition}

According to Noether's celebrated theorem, a Hamiltonian system that admits a continuous family of symmetries also admits a corresponding conserved quantity. Therefore one might expect that a Hamiltonian system that admits an approximate symmetry must also have an approximate conservation law. Kruskal showed that this is indeed the case for nearly-periodic Hamiltonian systems, as the following generalization of his argument shows.

\begin{proposition}\label{existence_of_mu}
Let $X_\epsilon$ be a nearly-periodic Hamiltonian system on the exact presymplectic manifold $(M,\Omega_\epsilon)$. Let $R_\epsilon$ be the associated roto-rate. There is a formal power series $\overline{\vartheta}_\epsilon = \overline{\vartheta}_0 + \epsilon\,\overline{\vartheta}_1 + \dots$ with coefficients in $\Omega^1(M)$ such that $\Omega_\epsilon = -\mathbf{d}\overline{\vartheta}_\epsilon$ and $\mathcal{L}_{R_\epsilon}\overline{\vartheta}_\epsilon = 0$. Moreover, the formal power series $\mu_\epsilon = \iota_{R_\epsilon}\overline{\vartheta}_\epsilon$ is a constant of motion for $X_\epsilon$ to all orders in perturbation theory. In other words,
\begin{align*}
\mathcal{L}_{X_\epsilon}\mu_\epsilon = 0,
\end{align*}
in the sense of formal power series.
\end{proposition}
\begin{proof}
To construct the $U(1)$-invariant primitive $\overline{\vartheta}_\epsilon$ we select an arbitrary primitive $\vartheta_\epsilon$ for $\Omega_\epsilon$ and set
\begin{align*}
\overline{\vartheta}_\epsilon = \frac{1}{2\pi}\int_0^{2\pi}\exp(\theta\mathcal{L}_{R_\epsilon})\vartheta_\epsilon\,d\theta.
\end{align*}
This formal power series satisfies $\mathcal{L}_{R_\epsilon}\overline{\vartheta}_\epsilon=0$ because
\begin{align*}
\mathcal{L}_{R_\epsilon}\overline{\vartheta}_\epsilon = \frac{1}{2\pi}\int_0^{2\pi}\frac{d}{d\theta}\exp(\theta\mathcal{L}_{R_\epsilon})\vartheta_\epsilon\,d\theta = 0.
\end{align*}
Moreover, since $\mathcal{L}_{R_\epsilon}\Omega_\epsilon = 0$ by Kruskal's Proposition \ref{existence_of_roto_rate}, we have
\begin{align*}
-\mathbf{d}\overline{\vartheta}_\epsilon = \frac{1}{2\pi}\int_0^{2\pi}\exp(\theta\mathcal{L}_{R_\epsilon})\Omega_\epsilon\,d\theta = \frac{1}{2\pi}\int_0^{2\pi}\Omega_\epsilon\,d\theta = \Omega_\epsilon,
\end{align*}
whence $\overline{\vartheta}_\epsilon$ is a primitive for $\Omega_\epsilon$.

To establish all-orders time-independence of $\mu_\epsilon = \iota_{R_\epsilon}\overline{\vartheta}_\epsilon$, we apply Cartan's formula and Corollary \ref{invariance_of_Hamiltonian} according to
\begin{align*}
\mathcal{L}_{X_\epsilon}\mu_\epsilon = \iota_{X_\epsilon}\mathbf{d}\iota_{R_\epsilon}\overline{\vartheta}_\epsilon  =  - \iota_{R_\epsilon}\iota_{X_\epsilon}\Omega_\epsilon = - \mathcal{L}_{R_\epsilon}H_\epsilon =  0. 
\end{align*}
\end{proof}

\begin{definition}\label{mu_defined}
The formal constant of motion $\mu_\epsilon$ provided by Proposition \ref{existence_of_mu} is the \textbf{adiabatic invariant} associated with a nearly-periodic Hamiltonian system.
\end{definition}

\section{Slow manifolds for nearly-periodic systems\label{sms_for_nps}}

Let $X_\epsilon$ be a smooth $\epsilon$-dependent vector field on a manifold $M$ equipped with an auxilliary Riemannian metric $g$. Without loss of generality, assume $X_\epsilon = O(1)$. An $N^{\text{th}}$-order \emph{slow manifold} $S_\epsilon\subset M$ for $X_\epsilon$ is an $\epsilon$-dependent submanifold such that the normal component of $X_\epsilon$ along $S_\epsilon$ is $O(\epsilon^{N+1})$ and the tangential component is $O(\epsilon)$. Note in particular that $S_0$ must be a manifold of fixed points for $X_0$. A slow manifold is \emph{elliptic} if the normal linearized dynamics for $X_0$ along $S_0$ are purely oscillatory.  

Intuitively, trajectories for $X_\epsilon$ that begin near an elliptic slow manifold $S_\epsilon$ should slowly drift along  $S_\epsilon$ (while rapidly oscillating around it) for some large time interval before possibly wandering away. By dropping the normal component $X^\perp_\epsilon$ of $X_\epsilon$ along $S_\epsilon$, one obtains a well-defined vector field $X_\epsilon^\parallel$ on $S_\epsilon$ that may be interpreted as a model for the slow drift dynamics. By way of this procedure, elliptic slow manifolds give rise to formal ``reduced" models in various scientific disciplines, especially in plasma physics, with its numerous multiscale models. The purpose of this Section is to prove that nearly-periodic systems always admit elliptic slow manifolds of every order.

Before proceeding, we must resolve a technical issue. The ``definition'' of slow manifolds given above is somewhat imprecise since $S_\epsilon$ moves as $\epsilon\rightarrow 0$. In order to eliminate this ambiguity, we introduce parameterizations.

\begin{definition}\label{parameterized_sm_def}
Let $X_\epsilon$ be a smooth $\epsilon$-dependent vector field on a manifold $M$ such that $X_0$ is not identically zero. An $N^{\text{th}}$-order \textbf{parameterized slow manifold} for $X_\epsilon$ is a smooth $\epsilon$-dependent embedding $\mathcal{S}_\epsilon:S_0\rightarrow M$ of some fixed manifold $S_0$ into $M$ with the following two properties.
\begin{itemize}
    \item[(1)] For each $s_0\in S_0$, $|X_\epsilon^\perp(\mathcal{S}_\epsilon(s_0))| = O(\epsilon^{N+1})$ as $\epsilon\rightarrow 0$.
    \item[(2)] For each $s_0\in S_0$, $|X_\epsilon^\parallel(\mathcal{S}_\epsilon(s_0))| = O(\epsilon)$ as $\epsilon\rightarrow 0$.
\end{itemize}
\end{definition}

Suppose $\varphi_\epsilon:S_0\rightarrow S_0$ is a smooth $\epsilon$-dependent diffeomorphism. If $\mathcal{S}_\epsilon$ is a parameterized slow manifold then $\mathcal{S}_\epsilon^\prime = \mathcal{S}_\epsilon\circ \varphi_\epsilon$ is also a parameterized slow manifold with the same image and of the same order. We may therefore precisely define an $N^{\text{th}}$-order \textbf{slow manifold} as a smooth $\epsilon$-dependent submanifold $S_\epsilon\subset M$ that is the image of some $N^{\text{th}}$-order parameterized slow manifold.

We will now suppose that $X_\epsilon$ is a nearly-periodic system, as defined in Section \ref{nps_background}, and proceed to construct slow manifolds. Our overarching strategy will be to find vector fields $X_\epsilon^{(N)}$ that agree with $X_\epsilon$ modulo terms that are $O(\epsilon^{N+1})$ and that possess genuine invariant submanifolds $S_\epsilon^{(N)}$. Then we will prove that certain open subsets of $S_\epsilon^{(N)}$ are $N^{\text{th}}$-order slow manifolds for $X_\epsilon$.

The motivation for this strategy comes from Kruskal's result (Proposition \ref{existence_of_normalizing_transformation}) on existence of normalizing transformations for nearly-periodic systems. In a formal sense, normalizing transformations provide coordinates that expose a hidden $U(1)$-symmetry underlying each nearly-periodic system. The connection between this observation and invariant manifold theory is that the set of $U(1)$-invariant points in phase space must be an invariant manifold for any dynamical system with $U(1)$-symmetry. The hitch in this argument, and the reason we will only find slow manifolds, rather than genuine invariant manifolds, is that Kruskal's normalizing transformations are only defined as formal power series. We will be forced to truncate these series, and in the process lose exact invariance. But since we may truncate at any order, that loss of invariance can be made arbitrarily small.

\begin{theorem}\label{sm_existence_thm}
Let $X_\epsilon$ be a nearly-periodic system on a manifold $M$. The associated limiting slow manifold, $S_0$, is a $0^{\text{th}}$-order slow manifold for $X_\epsilon$. Moreover, for each $N>0$ and each codimension-$0$ compact submanifold $C_0\subset M$, with or without boundary, there exists an $N^{\text{th}}$-order parameterized slow manifold $\mathcal{S}_\epsilon^{(N)}:\Sigma_0\rightarrow M$, for $X_\epsilon$, where $\Sigma_0 = S_0\cap \mathrm{int}\,C_0$.
\end{theorem}
\begin{remark}
It is worth highlighting two limiting cases of the theorem. (1) If the set of equilibrium points for $R_0$ is empty, then the theorem is vacuously true; the slow manifolds are merely empty sets. (2) If $M$ is compact, we may take $C_0 = M$ and conclude that $X_\epsilon$ admits a slow manifold diffeomorphic to $S_0$ at each order $N$.  For non-compact $M$, the slow manifolds provided by the theorem may fail to be diffeomorphic to $S_0$.    
\end{remark}

To prove Theorem \ref{sm_existence_thm}, we will use a pair of supporting Lemmas.
\begin{lemma}[See e.g. \cite{Kankaanrinta_2007}]\label{S0_is_submanifold}
The set of fixed points $Z\subset M$ of a $U(1)$-action on a manifold $M$ is an embedded submanifold.
\end{lemma}

\begin{proof}
Let $\Phi_\theta:M\rightarrow M$, $\theta\in U(1)$, be the $U(1)$-action, and let $R$ be the corresponding infinitesimal generator.
Assume that $M$ is equipped with a metric tensor $g$ that satisfies $\mathcal{L}_{R}g = 0$. Note that if $g$ is an arbitrary metric on $M$, then $\langle g\rangle = \tfrac{1}{2\pi}\int_0^{2\pi}\Phi_\theta^*g\,d\theta$ is a Riemannian metric that satisfies $\mathcal{L}_{R}\langle g\rangle  = 0$. Therefore our assumption introduces no loss of generality.

Suppose $m\in Z$ and let $\exp_m:T_mM\rightarrow M$ be the Riemannian exponential map at $m$. Since $g$ is $U(1)$-invariant, the geodesic flow on $TM$ commutes with $T\Phi_\theta$. It follows that the exponential map at $m$ intertwines the $U(1)$-actions $\Phi_\theta$ on $M$ and $L_\theta\equiv T_m\Phi_\theta$ on $T_mM$, i.e. $\exp_m\circ L_\theta = \Phi_\theta\circ \exp_m$ for each $\theta\in U(1)$. In light of this equivariance property and the inverse function theorem, we may therefore choose $U(1)$-invariant open subsets $U_m\subset M$ and $U_0\subset T_mM$, containing $m\in M$ and $0\in T_mM$ respectively, such that $\varphi_m = \exp_m\mid U_0:U_0\rightarrow U_m$ is a diffeomorphism. Since this diffeomorphism is equivariant, the preimage $Z_0$ of $Z_m\equiv Z\cap U_m$ under $\varphi_m$ must be equal to the fixed point set for the $U(1)$ action $L_\theta$ on $U_0$. Since $L_\theta$ is linear, $Z_0$ must be a linear subspace of $T_mM$. By restricting $\varphi_m^{-1}$ to $Z_0$, we therefore obtain a coordinate chart on $Z$ near $m$. Since $m\in Z$ is arbitrary, this shows that $Z$ is an embedded submanifold.
\end{proof}

\begin{lemma}\label{fixed_points_are_ims}
If $Y_\epsilon$ is a smooth $\epsilon$-dependent vector field on $M$ that commutes with the infinitesimal generator $R$ of a $U(1)$-action then the set of fixed points $Z$ for $R$ is an invariant submanifold for $Y_\epsilon$, for each $\epsilon$. 
\end{lemma}

\begin{proof}
Suppose $m\in Z$. We will show that the component of $Y_\epsilon$ normal to $Z$ vanishes, i.e.  $Y_\epsilon^\perp(m) = 0$. Let $w:M\rightarrow \mathbb{R}$ be a smooth $U(1)$-invariant bump function equal to $1$ near $m$ and $0$ outside of a compact set containing $m$. Let $F_t = \exp(t\,w\,Y_\epsilon)$ denote the flow map for $w\,Y$, and let $\Phi_\theta = \exp(\theta\,R)$ denote the $U(1)$-action generated by $R$. Since $[R,w\,Y_\epsilon] = 0$, we have $F_t\circ\Phi_\theta = \Phi_\theta\circ F_t$ for each $t\in\mathbb{R}, \theta\in U(1)$. Since $m$ is an equilibrium for $R$, we therefore have $F_t(m) = \Phi_\theta(F_t(m))$ for each $t,\theta$. In other words, $F_t(m)$ is a an equilibrium for $R$ for each $t$. By Lemma \ref{fixed_points_are_ims}, we then see that the parameterized curve $\gamma(t) =F_t(m) $ defines a smooth mapping from $\mathbb{R}$ into the submanifold $Z$. The curve's velocity at $t=0$ is therefore tangent to $Z$ at $\gamma(0) = F_0(m)=m$. But $d\gamma(0)/dt = w(m)\,Y_\epsilon(m) = Y_\epsilon(m)$, whence it follows that $Y_\epsilon(m)^\perp=  0$.
\end{proof}

\begin{proof}[proof of Theorem \ref{sm_existence_thm}]
First we prove that $S_0$ is a $0^{\text{th}}$-order slow manifold for $X_\epsilon$. By Lemma \ref{S0_is_submanifold}, we know $S_0$ is a submanifold. The mapping $\mathcal{S}_\epsilon^{(0)}:S_0\rightarrow M:s_0\mapsto s_0$ therefore defines a smooth embedding of $S_0$ into $M$. We claim $\mathcal{S}_0$ is a parameterized $0^{\text{th}}$-order slow manifold. To see this, note $R_0(\mathcal{S}_\epsilon^{(0)}(s_0)) = 0$ for each $s_0\in S_0$, which implies
\begin{align*}
    X_\epsilon(\mathcal{S}_\epsilon^{(0)}(s_0)) = \omega_0(\mathcal{S}_\epsilon^{(0)}(s_0))\,R_0(\mathcal{S}_\epsilon^{(0)}(s_0)) + O(\epsilon) = O(\epsilon),
\end{align*}
as claimed. It follows that $S_0 = \mathcal{S}_\epsilon^{(0)}(S_0)$ is a $0^{\text{th}}$-order slow manifold.

Next we prove the existence of higher-order slow manifolds. Fix an integer $N>0$ and a codimension-$0$ compact submanifold $C_0\subset M$, with or without boundary. By Prop. \ref{existence_of_normalizing_transformation}, there exists a normalizing transformation $\exp(\mathcal{L}_{G_\epsilon})$ for $X_\epsilon$ with generator $G_\epsilon$. We would like to construct a diffeomorphism $\Psi_\epsilon:M\rightarrow M$ that agrees with the formal diffeomorphism $\exp(G_\epsilon)$ to $N^{\text{th}}$-order, at least on $\text{int}\,C_0$. The simplest strategy for this construction would be to set $\Psi_\epsilon = \exp(\sum_{k=1}^N \epsilon^k\,G_k)$, but, when $M$ is non-compact, integral curves of the vector field $G_\epsilon^{(N)}=\sum_{k=1}^N \epsilon^k\,G_k$ may not exist for all time, and so the exponential $\exp(G_\epsilon^{(N)})$ may not exist either. To overcome this difficulty, we introduce the vector field $\mathcal{G}_\epsilon^{(N)} = w\,G_\epsilon^{(N)}$, where $w:M\rightarrow\mathbb{R}$ is a smooth function defined as follows. If $\partial C_0 = \emptyset$, so that $C_0$ is a union of connected components of $M$, $w=1$ on $C_0$, and $w=0$ otherwise. If $\partial C_0\neq \emptyset$, then we use the tubular neighborhood theorem to construct an increasing sequence of compact, codimension-$0$ submanifolds with boundaries $C_0\subset  C_0^{\prime}\subset C_0^{\prime\prime}$ such that $C_0\subset \text{int}\, C_0^{\prime}$ and $ C_0^{\prime}\subset \text{int}\,C_0^{\prime\prime}$, and define $w$ so that it satisfies $w=1$ in $ C_0^{\prime}$ and $w=0$ outside of $C_0^{\prime\prime}$. Thus, $w$, and therefore $\mathcal{G}_\epsilon^{(N)}$, have compact support. Since smooth vector fields with compact support have well-defined flow maps, we thereby obtain a smooth diffeomorphism $\Psi_\epsilon = \exp(\mathcal{G}_\epsilon^{(N)})$. This diffeomorphism agreees with $\exp(G_\epsilon)$ to $N^{\text{th}}$-order on $\text{int}\, C_0^{\prime}$ in the following sense. If $\bm{T}_\epsilon$ is any smooth $\epsilon$-dependent tensor field on $M$ then
\begin{align}
    \forall m\in \text{int}\, C_0^{\prime},\quad(\Psi_\epsilon^*\bm{T}_\epsilon)(m) &=\left(\bm{T}_\epsilon+\mathcal{L}_{\mathcal{G}_\epsilon^{(N)}}\bm{T}_\epsilon + \frac{1}{2}\mathcal{L}_{\mathcal{G}_\epsilon^{(N)}}^2\bm{T}_\epsilon + \dots\right)(m)\nonumber\\
    &=\left(\bm{T}_\epsilon+\mathcal{L}_{G_\epsilon^{(N)}}\bm{T}_\epsilon + \frac{1}{2}\mathcal{L}_{G_\epsilon^{(N)}}^2\bm{T}_\epsilon + \dots\right)(m)\nonumber\\
    & = \left(\bm{T}_\epsilon+\mathcal{L}_{G_\epsilon}\bm{T}_\epsilon + \frac{1}{2}\mathcal{L}_{G_\epsilon}^2\bm{T}_\epsilon + \dots\right)(m)+O(\epsilon^{N+1})\nonumber\\
    & = \left(\exp(\mathcal{L}_{G_\epsilon})\bm{T}_\epsilon\right)(m)+O(\epsilon^{N+1}),\label{Phi_approximates_nt}
\end{align}
in the sense of formal power series. Here we have used that $w=1$ on an open neighborhood of any $m\in\text{int}\,C_0^\prime $ to replace the Lie derivatives $\mathcal{L}_{\mathcal{G}_\epsilon^{(N)}}$ with $\mathcal{L}_{G_\epsilon^{(N)}}$.

Using the diffeomorphism $\Psi_\epsilon$, we now construct a vector field $X_\epsilon^{(N)}$ such that (a) $X_\epsilon^{(N)} = X_\epsilon + O(\epsilon^{N+1})$ in $\text{int}\,C_0^\prime $, and (b) $X_\epsilon^{(N)}$ admits an exact parameterized invariant manifold $\mathfrak{S}_\epsilon^{(N)}:S_0\rightarrow M$. Let $\overline{X}_\epsilon = \exp(-\mathcal{L}_{G_\epsilon})X_\epsilon = \overline{X}_0 + \epsilon\,\overline{X}_1+\epsilon^2\,\overline{X}_2 + \dots$, and set $\overline{X}_\epsilon^{(N)} = \sum_{k=1}^{N}\epsilon^{k}\,\overline{X}_k$. Since $\exp(-\mathcal{L}_{G_\epsilon})R_\epsilon = R_0$, where $R_\epsilon$ is the roto-rate for $X_\epsilon$, and $[R_\epsilon,X_\epsilon] = 0$ to all orders, each of the $\overline{X}_k$, and therefore $\overline{X}_\epsilon^{(N)}$, commutes with $R_0$. By Lemma \ref{fixed_points_are_ims}, it follows that the set of equilibrium points for $R_0$, i.e. $S_0$, is an invariant manifold for $\overline{X}_\epsilon^{(N)}$. We claim that the vector field $X_\epsilon^{(N)} = \Psi_\epsilon^*\overline{X}_\epsilon^{(N)}$ satisfies properties (a) and (b)
above. For (a), we use Eq.\,\eqref{Phi_approximates_nt}  to obtain
\begin{align*}
    \forall m\in\text{int}\,C_0^\prime\,,\quad X_\epsilon^{(N)}(m) &= (\exp(\mathcal{L}_{G_\epsilon})\overline{X}_\epsilon^{(N)})(m)+O(\epsilon^{N+1})\nonumber\\
    & = (\exp(\mathcal{L}_{G_\epsilon})\overline{X}_\epsilon)(m)+O(\epsilon^{N+1})\nonumber\\
    & = (\exp(\mathcal{L}_{G_\epsilon})\exp(-\mathcal{L}_{G_\epsilon})X_\epsilon)(m)+O(\epsilon^{N+1})\nonumber\\
    & = X_\epsilon(m) + O(\epsilon^{N+1}),
\end{align*}
in the sense of formal power series. For (b), we restrict the inverse of $\Psi_\epsilon$ to $S_0$ to obtain the embedding $\mathfrak{S}_\epsilon^{(N)} = \Psi_\epsilon^{-1}\mid S_0:S_0\rightarrow M$; since $S_0$ is an invariant manifold for $\overline{X}_\epsilon^{(N)}$, $\mathfrak{S}_\epsilon^{(N)}(S_0)$ is an invariant manifold for $X_\epsilon^{(N)}$.

To complete the proof, we first use $(X_\epsilon^{(N)})^\perp = 0$ along $\mathfrak{S}_\epsilon^{(N)}$ and $\mathfrak{S}_\epsilon^{(N)}(S_0\cap \text{int}\,C_0)\subset C_0^\prime$ for sufficiently-small $\epsilon$ to obtain
\begin{align}
    \forall s_0\in S_0\cap\text{int}\,C_0\,,\quad |X_\epsilon^\perp(\mathfrak{S}_\epsilon^{(N)}(s_0))| &= |(X_\epsilon^{(N)})^{\perp}(\mathfrak{S}_\epsilon^{(N)}(s_0))|+O(\epsilon^{N+1})\nonumber\\
    & = 0+O(\epsilon^{N+1}).\label{small_perp}
\end{align}
Then we note that $X_\epsilon^{(N)} = X_0 +O(\epsilon) = \omega_0\,R_0+O(\epsilon)$ since $\Psi_\epsilon$ is near-identity, which implies in particular
\begin{align}
    \forall s_0\in S_0\cap\text{int}\,C_0\,,\quad |(X_\epsilon^{(N)})^\parallel(\mathfrak{S}_\epsilon^{(N)}(s_0))| &=O(\epsilon)  .\label{small_parallel}
\end{align}
Equations \,\eqref{small_perp}-\eqref{small_parallel} say $\mathcal{S}_\epsilon^{(N)} = \mathfrak{S}_\epsilon^{(N)}\mid S_0\cap \text{int}\,C_0$ is an $N^{\text{th}}$-order parameterized slow manifold.
\end{proof}

\section{Normal stability in nearly-periodic Hamiltonian systems}
Given an invariant manifold, or more generally an almost invariant manifold like those provided by Theorem \ref{sm_existence_thm}, it is important to understand the stability of nearby trajectories. In other words, if a trajectory begins near such an object in phase space, then how long will it remain nearby? The answer to this question sets limits on model reduction strategies based on projecting to the (almost) invariant object. When all nearby trajectories remain nearby on some time interval $\mathcal{I}$, projecting should provide a reasonable reduced model for dynamics near the manifold for times in $\mathcal{I}$. We say that the manifold is \emph{normally-stable} on $\mathcal{I}$.  However, after a trajectory's transverse displacement becomes large, the projected dynamics may have very little to do with the true dynamics. The purpose of this Section is to establish a useful tool for establishing normal stability of slow manifolds in nearly-periodic systems that admit a particular kind of Hamiltonian structure.

A nearly-periodic system $X_\epsilon$ can exhibit a Hamiltonian structure in various ways. In certain cases, $X_\epsilon$ may be Hamiltonian with respect to an $\epsilon$-independent symplectic form $\Omega$ on $M$, meaning there is a smooth $\epsilon$-dependent function $H_\epsilon:M\rightarrow\mathbb{R}$ such that $\iota_{{X}_\epsilon}\Omega = \mathbf{d}H_\epsilon$. Examples include $2$-degree-of-freedom canonical Hamiltonian systems with Hamiltonians of the form $H(q_1,p_1,q_2,p_2) = \tfrac{1}{2}(q_1^2+p_1^2) +\epsilon\,U(q_1,p_1,q_2,p_2)$, where $U$ is any smooth function. More generally, $X_\epsilon$ may be Hamiltonian with respect to a smooth $\epsilon$-dependent symplectic form $\Omega_\epsilon$ whose singular limit $\Omega_0$ is merely pre-symplectic. In other words, $\Omega_0$ is closed, but may be degenerate as a $2$-form. We call such $\epsilon$-dependent $2$-forms \emph{barely-symplectic forms}. The Lorentz force equations describing the motion of a charged particle in a strong magnetic field exhibit a barely-symplectic structure. Since the Hamiltonian structure for a single charged particle frequently appears in the infinite-dimensional Hamiltonian structures underlying non-dissipative models of plasma dynamics (CITE), barely-symplectic Hamiltonian structures play an important role in plasma physics.

Motivated by the above remarks, we will focus our attention on normal stability of slow manifolds arising in nearly-periodic Hamiltonian system on barely-symplectic manifolds. (Defined below.) Since our methods involve perturbation theory, we will focus in particular on barely-symplectic forms that exhibit at worst a \emph{regular singularity} as $\epsilon\rightarrow 0$. The class of barely-symplectic forms with regular singularities at $\epsilon=0$ appears to be broad enough to cover many significant applications, including those discussed in Section \ref{gc_application_section}.

\subsection{Barely-symplectic manifolds\label{bsm_section}}
Regular barely-symplectic manifolds provide the differential-geometric setting for our results on slow manifold normal stability. The purpose of this subsection is to precisely define and outline the basic properties of these manifolds. The notion of barely-symplectic manifold should be contrasted with the relatetd notion of \emph{folded symplectic manifold} introduced in \cite{Silva_2000}. Rather than exhibiting singularities as a parameter tends to a limiting value, folded symplectic forms degenerate on hypersurfaces in phase space.

\begin{definition}
A \textbf{barely-symplectic form} on a manifold $M$ is a smooth $\epsilon$-dependent $2$-form $\Omega_\epsilon$ such that 
\begin{itemize}
    \item $\mathbf{d}\Omega_\epsilon = 0$ for each $\epsilon\in\mathbb{R}$
    \item $\Omega_\epsilon$ is symplectic whenever $\epsilon\neq 0$
\end{itemize}
The barely-symplectic form $\Omega_\epsilon$ is \textbf{exact} if $\Omega_\epsilon = -\mathbf{d}\vartheta_\epsilon$ for some smooth $\epsilon$-dependent $1$-form $\vartheta_\epsilon$. A(n) (\textbf{exact}) \textbf{barely-symplectic manifold} is a pair $(M,\Omega_\epsilon)$, where $M$ is a manifold and $\Omega_\epsilon$ is a(n) (exact) barely-symplectic form on $M$. A \textbf{Hamiltonian system} on a barely-symplectic manifold is a smooth $\epsilon$-dependent vector field $X_\epsilon$ such that $\iota_{X_\epsilon}\Omega_\epsilon = \mathbf{d}H_\epsilon$, for some smooth $\epsilon$-dependent function $H_\epsilon$, referred to as the system's \textbf{Hamiltonian}.
\end{definition}

Since a barely-symplectic form is non-degenerate for non-zero $\epsilon$, each barely-symplectic form induces an $\epsilon$-dependent Poisson structure with a possible singularity at $\epsilon = 0$. We may therefore classify barely-symplectic forms according to the severity of that singularity.

\begin{definition}
Let $\Omega_\epsilon$ be a barely-symplectic form and let $\mathcal{J}_{\epsilon}$ be the $\epsilon$-dependent Poisson bivector defined for $\epsilon\neq 0$ by inverting $\Omega_\epsilon$. We say $\Omega_\epsilon$ is \textbf{regular} if there is some non-negative integer $d$ such that $\mathcal{J}_\epsilon = \epsilon^{-d}j_\epsilon$, where $j_\epsilon$ is a smooth $\epsilon$-dependent bivector. The smallest $d$ such that $\mathcal{J}_\epsilon = \epsilon^{-d}j_\epsilon$ is called the \textbf{degeneracy index} of $\Omega_\epsilon$. When no such $d$ exists, $\Omega_\epsilon$ is \textbf{irregular}.
\end{definition}

\begin{example}
Let $(M_1,\Omega_1)$, $(M_2,\Omega_2)$ be a pair of symplectic manifolds, and let $f_1(\epsilon),f_2(\epsilon)$ be a pair of smooth functions, each with at most an isolated zero at $\epsilon=0$. If we equip the product manifold $M = M_1\times M_2$ with the $\epsilon$-dependent $2$-form $\Omega_\epsilon = f_1(\epsilon)\,\pi_1^*\Omega_1+f_2(\epsilon)\,\pi_2^*\Omega_2$, where $\pi_k:M\rightarrow M_k$ denotes projection onto the $k^{\text{th}}$ factor, we obtain a barely-symplectic manifold $(M,\Omega_\epsilon)$. 

When $f_1(\epsilon) = \epsilon^n$ and $f_2(\epsilon) = \epsilon^m$ with $n\geq m\geq 0$, $\Omega_\epsilon$ is a regular barely-symplectic form. Since the Poisson bivector associated with $\Omega_\epsilon$ is given by $\mathcal{J}_\epsilon = \epsilon^{-n}\mathcal{J}_1 + \epsilon^{-m}\mathcal{J}_2 = \epsilon^{-n}(\mathcal{J}_1 + \epsilon^{n-m}\mathcal{J}_2)$, the degeneracy index of $\Omega_\epsilon$ is $d = n$. To obtain an irregular barely-symplectic form, we may set $f_1(\epsilon) = \exp(-1/\epsilon^2)$, $f_2(\epsilon) = 1$. Then $\mathcal{J}_\epsilon = \exp(1/\epsilon^2)\,\mathcal{J}_1 + \mathcal{J}_2 = \exp(1/\epsilon^2)(\mathcal{J}_1+\exp(-1/\epsilon^2)\mathcal{J}_2)$. Although $\mathcal{J}_\epsilon = j_\epsilon/f_1(\epsilon)$ with $j_\epsilon$ smooth, $1/f_1(\epsilon)$ tends to $\infty$ as $\epsilon\rightarrow 0$ so quickly that the singularity cannot be tamed by any power of $\epsilon$.
\end{example}

Although a barely-symplectic form $\Omega_\epsilon$ may fail to be symplectic when $\epsilon=0$, regular $\Omega_\epsilon$ with degeneracy index $d$ nevertheless behave much like ordinary, $\epsilon$-independent symplectic forms. Heuristically, one just needs to include the first $d+1$ terms in the power series expansion $\Omega_\epsilon = \Omega_0+\epsilon\,\Omega_1 + \dots$, rather than $\Omega_0$ by itself, to employ to symplectic methods. As an illustration of this heuristic, we have the following version of Darboux's theorem for regular barely-symplectic manifolds.

\begin{proposition}[Darboux theorem for regular barely-symplectic manifolds]\label{barely_symplectic_darboux}
Let $(M,\Omega_\epsilon)$ be a compact, regular, barely-symplectic manifold with degeneracy index $d$. Assume the de Rham cohomology class $[\Omega_\epsilon]$ defined by the barely-symplectic form satisfies $[\Omega_\epsilon] = [\Omega_0]$ for each $\epsilon$, and set $\Omega_\epsilon^{(d)} = \sum_{k=0}^{d}\epsilon^k\,\Omega_k$, where $\Omega_k$ is the $k^{\text{th}}$ coefficient in the power series expansion $\Omega_\epsilon = \Omega_0 + \epsilon\,\Omega_1 + \dots$. There exists an $\epsilon_0>0$ and a smooth $\epsilon$-dependent diffeomorphism $\Psi_\epsilon:M\rightarrow M$, $\epsilon\in[-\epsilon_0,\epsilon_0]$, such that $\Psi_\epsilon^*\Omega_\epsilon = \Omega_\epsilon^{(d)}$.

\end{proposition}

\begin{proof}
Define $\Omega_\epsilon^\lambda = [1-\lambda]\,\Omega_\epsilon + \lambda\,\Omega_\epsilon^{(d)}$, $\lambda\in[0,1]$. We would like to establish non-degeneracy of $\Omega_\epsilon^\lambda$ for sufficiently-small $\epsilon$. Notice that $\Omega_\epsilon^\lambda =  \Omega_\epsilon +\lambda[\Omega_\epsilon^{(d)}-\Omega_\epsilon] = \Omega_\epsilon + O(\lambda\epsilon^{d+1})$. Therefore, if $\mathcal{J}_\epsilon = \epsilon^{-d}j_\epsilon$ denotes the Poisson bivector induced by $\Omega_\epsilon$, we have
\begin{align*}
    -\widehat{\mathcal{J}}_\epsilon\,\widehat{\Omega}_\epsilon^\lambda = \text{id}_{TM} + \lambda\,\epsilon\,\widehat{\psi}_\epsilon,
\end{align*}
where $\epsilon\,\widehat{\psi}_\epsilon = -\epsilon^{-d}\widehat{j}_\epsilon [\widehat{\Omega}_\epsilon^{(d)} - \widehat{\Omega}_\epsilon]$ is a smooth $\epsilon$-dependent bundle map $TM\rightarrow TM$. By openness of the set of invertible matrices and compactness of $M$, there is therefore some $\epsilon_0>0$ such that $\text{id}_{TM} + \lambda\,\epsilon\,\psi_\epsilon$ is invertible for $\epsilon\in [-\epsilon_0,\epsilon_0]$, with smooth $\epsilon$-dependent inverse. This implies $\widehat{\Omega}_\epsilon^\lambda$ is invertible for non-zero $\epsilon\in[-\epsilon_0,\epsilon_0]$ with inverse given by $(\widehat{\Omega}_\epsilon^\lambda)^{-1} =-\widehat{\chi}_\epsilon^\lambda\widehat{\mathcal{J}}_\epsilon  = -\epsilon^{-d}\widehat{\chi}_\epsilon^\lambda\,\widehat{j}_\epsilon$, where $\widehat{\chi}_\epsilon^\lambda = [\text{id}_{TM} + \lambda\,\epsilon\,\widehat{\psi}_\epsilon]^{-1}$.

Next we will derive a useful formula for the difference $\Omega_\epsilon^{(d)} - \Omega_\epsilon$. By the cohomological condition $[\Omega_\epsilon] = [\Omega_0]$, we have $[\Omega_\epsilon - \Omega_0] = 0$, which implies that there is a smooth $\epsilon$-dependent $1$-form $\vartheta_\epsilon$ such that $\Omega_\epsilon = \Omega_0 -\mathbf{d}\vartheta_\epsilon$. Let $\vartheta_\epsilon = \vartheta_0 + \epsilon\,\vartheta_1 + \dots$ be that form's formal power series expansion, and set $\vartheta_\epsilon^{(d)} = \sum_{k=0}^{d}\epsilon^k\,\vartheta_k$. By equality of Taylor series, we have $\mathbf{d}\vartheta_0 = 0$ and $\Omega_\epsilon^{(d)} = \Omega_0 -\mathbf{d}\vartheta_\epsilon^{(d)}$. In particular we have the useful identity 
\begin{align}
\Omega_\epsilon^{(d)} - \Omega_\epsilon = \Omega_0 - \mathbf{d}\vartheta_\epsilon^{(d)} - \Omega_\epsilon = \mathbf{d}(\vartheta_\epsilon -\vartheta_\epsilon^{(d)}).\label{ch_cons}
\end{align}

Finally, we will construct the diffeomorphism $\Psi_\epsilon$. For $\epsilon\in[-\epsilon_0,\epsilon_0]$, define $G_\epsilon^\lambda = (\widehat{\Omega}_\epsilon^{\lambda})^{-1}\,(\vartheta_\epsilon^{(d)}-\vartheta_\epsilon)$. Since $\vartheta_\epsilon^{(d)}-\vartheta_\epsilon = O(\epsilon^{d+1})$ and $(\Omega_\epsilon^{\lambda})^{-1}=O(\epsilon^{-d})$, $G_\epsilon^\lambda$ depends smoothly on both $\epsilon$ and $\lambda$. If $F_\lambda^\epsilon$ denotes the $t=\lambda$ flow map for $G_\epsilon^\lambda$, we therefore have
\begin{align}
    \frac{d}{d\lambda}(F_\lambda^\epsilon)^* \Omega_\epsilon^\lambda &= (F_\lambda^\epsilon)^*\mathbf{d}\left(\iota_{G_\epsilon^\lambda}\Omega_\epsilon^\lambda + \vartheta_\epsilon- \vartheta_\epsilon^{(d)}\right)\nonumber\\
    & = (F_\lambda^\epsilon)^*(\widehat{\Omega}_\epsilon^\lambda\,G_\epsilon^\lambda + \vartheta_\epsilon- \vartheta_\epsilon^{(d)})\nonumber\\
    & =0,\nonumber
\end{align}
where we used the formula \eqref{ch_cons} on the first line.
This proves the theorem with $\Psi_\epsilon = (F^\epsilon_1)^{-1}$ since it implies $(F_1^\epsilon)^*\Omega_\epsilon^{(d)} = \Omega_\epsilon$.

\end{proof}

While we will not use the barely-symplectic Darboux theorem directly in this Article, it will be useful in what follows to distill the theorem's essential ingredient into the following Lemma.
\begin{lemma}\label{Darboux_lemma}
Let $M$ be a smooth manifold, and let $j_\epsilon$ and $\Omega_\epsilon$ be formal power series with bivector and $2$-form coefficients, respectively. Assume there exists a non-negative integer $d$ such that $-\widehat{j}_\epsilon\,\widehat{\Omega}_\epsilon = \epsilon^{d}\,\mathrm{id}_{TM}$, in the sense of formal power series. Then the smooth $\epsilon$-dependent $2$-form $\Omega_\epsilon^{(d)} = \sum_{k=0}^d\epsilon^{k}\,\Omega_k$ has the following properties.
\begin{itemize}
    \item[(1)] Given a compact set $C\subset M$ there exists an $\epsilon_0>0$ such that, for each non-zero $\epsilon\in[-\epsilon_0,\epsilon_0]$, $\Omega_\epsilon^{(d)}$ is non-degenerate on $C$. Moreover, $(\widehat{\Omega}_\epsilon^{(d)})^{-1} = -\epsilon^{-d}\,\widehat{\kappa}_\epsilon$, where $\kappa_\epsilon$ is a $0^{\mathrm{th}}$-order approximation of $j_\epsilon$.
    \item[(2)] There exists a formal power series $\kappa_\epsilon$ with bivector coefficients such that $-\epsilon^{-d}\widehat{\kappa}_\epsilon$ is a formal inverse for $\widehat{\Omega}_\epsilon^{(d)}$.
\end{itemize}
\end{lemma}
\begin{remark}
The $\kappa_\epsilon$ in property (2) above is not in general equal to the $\kappa_\epsilon$ in property $(1)$; the former is merely a formal power series with smooth bivector coefficients, while the latter is a smooth $\epsilon$-dependent bivector defined on the compact set $C$. However, within $C$, the power series expansion of $\kappa_\epsilon$ from (1) agrees with $\kappa_\epsilon$ from $(2)$, which justifies the abuse of notation.
\end{remark}

\begin{proof}
Let $\Delta\Omega_\epsilon = \epsilon^{-(d+1)}\,(\Omega_\epsilon - \Omega_\epsilon^{(d)})$. By assumption, $\epsilon^d\text{id}_{TM} = -\widehat{j}_\epsilon (\widehat{\Omega}_\epsilon^{(d)} + \epsilon^{d+1}\Delta\widehat{\Omega}_\epsilon)$, which may also be written
\begin{align*}
    \epsilon^d(\text{id}_{TM}+\epsilon\,\widehat{j}_\epsilon\,\Delta\widehat{\Omega}_\epsilon) = -\widehat{j}_\epsilon \,\widehat{\Omega}_\epsilon^{(d)}.
\end{align*}
It follows that the formal power series $\widehat{\kappa}_\epsilon$ with bundle-map  coefficients ($T^*M\rightarrow TM$) defined by
\begin{align}
    \widehat{\kappa}_\epsilon& = (\text{id}_{TM} + \epsilon\,\widehat{j}_\epsilon\,\Delta\widehat{\Omega}_\epsilon)^{-1}\,\widehat{j}_\epsilon\nonumber\\
    & = (\text{id}_{TM} - \epsilon\,\widehat{j}_\epsilon\,\Delta\widehat{\Omega}_\epsilon + \epsilon^2\,[\widehat{j}_\epsilon\,\Delta\widehat{\Omega}_\epsilon]^2 + \dots)\,\widehat{j}_\epsilon\label{kappa_formula}
\end{align}
satisfies $-\widehat{\kappa}_\epsilon\,\widehat{\Omega}_\epsilon^{(d)} = \epsilon^d\,\text{id}_{TM}$, in the sense of formal power series. We will therefore establish property (2) if we can demonstrate that $\widehat{\kappa}_\epsilon$ is the bundle map associated with some formal power series $\kappa_\epsilon$ with bivector coefficients. Equivalently, we must show $\widehat{\kappa}_\epsilon^* = -\widehat{\kappa}_\epsilon$, where $\widehat{\kappa}_\epsilon^*$ denotes the dual of $\widehat{\kappa}_\epsilon$. For this we compute directly:
\begin{align*}
    \widehat{\kappa}_\epsilon^*& = -\widehat{j}_\epsilon\,(\text{id}_{T^*M} + \epsilon\,\Delta\widehat{\Omega}_\epsilon\,\widehat{j}_\epsilon)^{-1}\nonumber\\
    & = -(\text{id}_{TM} - \epsilon\,\widehat{j}_\epsilon\,\Delta\widehat{\Omega}_\epsilon + \epsilon^2\,[\widehat{j}_\epsilon\,\Delta\widehat{\Omega}_\epsilon]^2 + \dots)\widehat{j}_\epsilon\nonumber\\
    & = -\widehat{\kappa}_\epsilon,
\end{align*}
where we have used the identity $\widehat{j}_\epsilon[\Delta\widehat{\Omega}_\epsilon\,\widehat{j}_\epsilon]^n = [\widehat{j}_\epsilon\,\Delta\widehat{\Omega}_\epsilon]^n\,\widehat{j}_\epsilon$ for each non-negative integer $n$. We conclude that $\Omega_\epsilon^{(d)}$ satisfies property (2).

Next we establish property (1). Define the smooth $\epsilon$-dependent bivector $\kappa_\epsilon^{(d)} = \sum_{k=0}^d\,\epsilon^k\,\kappa_k$. The formal power series identity $-\widehat{\kappa}_\epsilon\,\widehat{\Omega}_\epsilon^{(d)} = \epsilon^d\,\text{id}_{TM}$ implies that the the Taylor expansion of the smooth $\epsilon$-dependent bundle map $-\widehat{\kappa}_\epsilon^{(d)}\,\widehat{\Omega}_\epsilon^{(d)}$ is given by $-\widehat{\kappa}_\epsilon^{(d)}\,\widehat{\Omega}_\epsilon^{(d)} = \epsilon^d\,\text{id}_{TM} + O(\epsilon^{d+1})$. Taylor's theorem with remainder therefore implies that there is a smooth $\epsilon$-dependent bundle map $\widehat{\psi}_\epsilon:TM\rightarrow TM$ such that $-\widehat{\kappa}_\epsilon^{(d)}\,\widehat{\Omega}_\epsilon^{(d)} = \epsilon^d(\text{id}_{TM}+\epsilon\,\widehat{\psi}_\epsilon)$. Given a compact set $C\subset M$, we may choose $\epsilon_0$ small enough so that $(\text{id}_{TM}+\epsilon\,\widehat{\psi}_\epsilon)$ is invertible on $C$ when $\epsilon\in[-\epsilon_0,\epsilon_0]$. Thus, in $C$ and for non-zero $\epsilon\in[-\epsilon_0,\epsilon_0]$, we have
\begin{align*}
    (\widehat{\Omega}_\epsilon^{(d)})^{-1} = -\epsilon^{-d}(\text{id}_{TM} + \epsilon\,\widehat{\psi}_\epsilon)^{-1}\,\widehat{\kappa}_\epsilon^{(d)},
\end{align*}
as claimed.
\end{proof}

As an immediate application of Lemma \ref{Darboux_lemma}, we will prove the following refinement of Kruskal's result \ref{existence_of_normalizing_transformation} for regular barely-symplectic manifolds.

\begin{proposition}\label{KL_generator}
Let $X_\epsilon$ be a nearly-periodic Hamiltonian system on a barely-symplectic manifold $(M,\Omega_\epsilon)$. Assume that $(M,\Omega_\epsilon)$ is exact and regular with degeneracy index $d$. There exists a normalizing transformation for $X_\epsilon$ with generator $K_\epsilon$ such that the formal power series $\overline{\Omega}_\epsilon = \exp(-\mathcal{L}_{K_\epsilon})\Omega_\epsilon$ truncates at $O(\epsilon^d)$:
\begin{align*}
    \overline{\Omega}_\epsilon = -\mathbf{d}\overline{\Theta}_0 - \epsilon\,\mathbf{d}\overline{\Theta}_1-\dots-\epsilon^d\,\mathbf{d}\overline{\Theta}_d,
\end{align*}
where each $\overline{\Theta}_k$ satisfies $\mathcal{L}_{R_0}\overline{\vartheta}_k = 0$.
\end{proposition}

\begin{proof}
Let $R_\epsilon$ be the roto-rate for $X_\epsilon$. By Proposition \ref{existence_of_mu}, there is a formal power series $\overline{\vartheta}_\epsilon$ such that $\Omega_\epsilon = -\mathbf{d}\overline{\vartheta}_\epsilon$ and $\mathcal{L}_{R_\epsilon}\overline{\vartheta}_\epsilon = 0$. By Theorem \ref{existence_of_normalizing_transformation}, there exists a normalizing transformation for $X_\epsilon$ with generator $G_\epsilon$. Applying $\exp(-\mathcal{L}_{G_\epsilon})$ to the identities $\mathcal{L}_{R_\epsilon}\overline{\vartheta}_\epsilon = 0$ and $\Omega_\epsilon = -\mathbf{d}\overline{\vartheta}_\epsilon$ therefore implies $\mathcal{L}_{R_0}\overline{\theta}_\epsilon = 0$ and $\omega_\epsilon = -\mathbf{d}\overline{\theta}_\epsilon$, where $\overline{\theta}_\epsilon = \exp(-\mathcal{L}_{G_\epsilon})\overline{\vartheta}_\epsilon$ and $\omega_\epsilon = \exp(-\mathcal{L}_{G_\epsilon})\Omega_\epsilon$.

We would like to construct a Lie transform with generator $A_\epsilon$ such that 
\begin{align}
    [A_\epsilon,R_0] & = 0\\
    \exp(-\mathcal{L}_{A_\epsilon})\omega_\epsilon &= -\mathbf{d}(\overline{\theta}_0 + \epsilon\,\overline{\theta}_1+\dots+\epsilon^d\,\overline{\theta}_d).\label{truncated_target}
\end{align}
If such an $A_\epsilon$ exists then the Baker-Campbell-Hausdorf (BCH) formula implies $\exp(\mathcal{L}_{G_\epsilon})\exp(\mathcal{L}_{A_\epsilon})$ is a normalizing transformation with the desired properties. We will construct $\exp(\mathcal{L}_{A_\epsilon})$ as the composition of a sequence of Lie transforms $\exp(\mathcal{L}_{a_\epsilon^{(k)}})$ with generators $a_\epsilon^{(k)}$.

Let $\exp(\mathcal{L}_{a_\epsilon^{(1)}})$ be a Lie transform with generator $a_\epsilon^{(1)}$. Applying $\exp(-\mathcal{L}_{a_\epsilon^{(1)}})$ to $\overline{\theta}_\epsilon$ gives
\begin{align*}
    \exp(-\mathcal{L}_{a_\epsilon^{(1)}})\overline{\theta}_\epsilon & = \overline{\theta}_\epsilon -\mathcal{L}_{a_\epsilon^{(1)}}\overline{\theta}_\epsilon + \dots\nonumber\\
    & = \overline{\theta}_\epsilon -\iota_{a_\epsilon^{(1)}}\mathbf{d}\overline{\theta}_\epsilon - \mathbf{d}\iota_{a_\epsilon^{(1)}}\overline{\theta}_\epsilon + \dots,
\end{align*}
where we have applied Cartan's formula for the Lie derivative. Suppose $a_\epsilon^{(1)}$ is chosen such that
\begin{align}
    \epsilon^{d+1}\,\overline{\theta}_{d+1} - \iota_{a_\epsilon^{(1)}}\mathbf{d}(\overline{\theta}_0+\dots+\epsilon^d\,\overline{\theta}_d) = 0.\label{homological_eqn}
\end{align}
Then we would have $\exp(-\mathcal{L}_{a_\epsilon^{(1)}})\mathbf{d}\overline{\theta}_\epsilon =\mathbf{d}( \overline{\theta}_0+\dots+\epsilon^d\,\overline{\theta}_d )+ O(\epsilon^{d+2})$, which is one step closer to our target \eqref{truncated_target}. So let us assess the solvability of Eq.\,\eqref{homological_eqn}.

With ${\omega}_\epsilon^{(d)} = -\mathbf{d}(\overline{\theta}_0+\dots+\epsilon^d\,\overline{\theta}_d)$, Eq.\,\eqref{homological_eqn} reads
\begin{align}
    \widehat{\omega}_\epsilon^{(d)}a_\epsilon^{(1)} = - \epsilon^{d+1}\,\overline{\theta}_{d+1}.\label{op_v}
\end{align}
Since $\Omega_\epsilon$ is a regular barely-symplectic form, there is a smooth $\epsilon$-dependent bivector $j_\epsilon$ such that $\mathcal{J}_\epsilon = \epsilon^{-d}j_\epsilon$ inverts $\Omega_\epsilon$, i.e. $-\epsilon^{-d}\widehat{j}_\epsilon\,\widehat{\Omega}_\epsilon = \text{id}_{TM}$. Applying the Lie transform $\exp(-\mathcal{L}_{G_\epsilon})$ to this identities gives
\begin{align}
    -\widehat{b}_\epsilon\,\widehat{\omega}_\epsilon & = \epsilon^d\text{id}_{TM}\label{tan_identity}
\end{align}
where $b_\epsilon = \exp(-\mathcal{L}_{G_\epsilon})j_\epsilon$. Lemma \ref{Darboux_lemma} therefore implies that there is a formal power series $\kappa_\epsilon$ with bivector coefficients such that $-\epsilon^{-d}\widehat{\kappa}_\epsilon$ is a formal inverse for $\widehat{\omega}_\epsilon^{(d)}$. Applying this formal inverse to both sides of Eq.\,\eqref{op_v} now gives a formula for $a_\epsilon^{(1)}$ with the desired properties:
\begin{align*}
    a_\epsilon^{(1)} = \epsilon\, \widehat{\kappa}_\epsilon\,\overline{\theta}_{d+1}.
\end{align*}
Note that since $\overline{\vartheta}_{d+1}$ and $\omega_\epsilon^{(d)}$ are each $R_0$-invariant, the formula \eqref{kappa_formula} implies this $a_\epsilon^{(1)}$ satisfies $[a_\epsilon^{(1)},R_0] = 0$.

Modulo exact $1$-forms, we now have $\exp(-\mathcal{L}_{a_\epsilon^{(1)}})\overline{\theta}_\epsilon =\overline{\theta}_\epsilon^{(d)} + O(\epsilon^{d+2}) $, where the higher-order terms not displayed are each $R_0$-invariant. Using the same procedure used to find $a_\epsilon^{(1)}$, we may now construct an $a_\epsilon^{(2)}$ such that $\exp(-\mathcal{L}_{a_\epsilon^{(2)}})\exp(-\mathcal{L}_{a_\epsilon^{(1)}})\overline{\theta}_\epsilon = \overline{\theta}_\epsilon^{(d)} + O(\epsilon^{d+3})$ modulo exact $1$-forms, where again the higher-order terms are $R_0$-invariant, and $[a_\epsilon^{(2)},R_0] =0$. Repeating this construction \emph{ad infinitum} produces a sequence of $R_0$-invariant $a_\epsilon^{(n)}$ such that $\dots\exp(-\mathcal{L}_{a_\epsilon^{(3)}})\exp(-\mathcal{L}_{a_\epsilon^{(2)}})\exp(-\mathcal{L}_{a_\epsilon^{(1)}})\overline{\theta}_\epsilon = \overline{\theta}_\epsilon^{(d)}+O(\epsilon^{\infty})$ modulo exact $1$-forms. The BCH formula therefore gives us an $A_\epsilon$ defined by $\exp(-\mathcal{L}_{A_\epsilon}) = \dots\exp(-\mathcal{L}_{a_\epsilon^{(3)}})\exp(-\mathcal{L}_{a_\epsilon^{(2)}})\exp(-\mathcal{L}_{a_\epsilon^{(1)}})$ with the desired properties.

\end{proof}

\begin{example}
Dynamics of a non-relativistic charged particle in a strong magnetic field are described by the Lorentz force system $\dot{\bm{v}} = \bm{v}\times\bm{B}(\bm{x})$, $\dot{\bm{x}} = \epsilon\,\bm{v}$ on $M=\mathbb{R}^3\times\mathbb{R}^3\ni (\bm{x},\bm{v})$. Here $\bm{B} = \nabla\times\bm{A}$ denotes the magnetic field. Assuming $|\bm{B}|$ is nowhere vanishing, this system comprises a Hamiltonian  nearly-periodic system on the regular, exact, barely symplectic manifold $(M,-\mathbf{d}\vartheta_\epsilon)$, where $\vartheta_\epsilon = \bm{A}\cdot d\bm{x} + \epsilon\,\bm{v}\cdot d\bm{x}$ and $H_\epsilon = \epsilon^2\,\tfrac{1}{2}\,|\bm{v}|^2$.

Particles that move under the influence of the Lorentz force rapidly gyrate around magnetic field lines, while drifting relatively slowly along and across them. The slow drifts are described by the so-called guiding center theory. In Refs.\,\cite{Littlejohn_1981,Littlejohn_1982}, R. L. Littlejohn devised a method of computing normalizing transformations for the Lorentz force system that exposed the Hamiltonian structure underlying guiding center dynamics for the first time. As explained in Ref.\,\cite{Burby_gc_2013}, arbitrary choices inherent to Littlejohn's method of selecting the generator $G_\epsilon$ of the normalizing transformation may be performed to ensure
\begin{align*}
    \exp(-\mathcal{L}_{G_\epsilon})\vartheta_\epsilon &= \bm{A}\cdot d\bm{x} + \epsilon\, (\bm{v}\cdot\bm{b})\bm{b}\cdot d\bm{x}\nonumber\\
    &+\epsilon^2\,\frac{1}{2|\bm{B}|}\left(\bm{v}\times\bm{b}\cdot d\bm{v} - (\bm{v}\cdot\bm{b})[\nabla\bm{b}\cdot\bm{v}\times\bm{b}]\cdot d\bm{x}\right) + O(\epsilon^{\infty})
\end{align*}
modulo exact $1$-forms. Since the degeneracy index for $-\mathbf{d}\vartheta_\epsilon$ is $d = 2$, this truncation could have been predicted by Proposition \ref{KL_generator}.
\end{example}

We close this section by noting that irregular barely-symplectic forms may degenerate so rapidly as $\epsilon\rightarrow 0$ that symplectic methods do not apply to them. We are currently unfamiliar with any physical examples of such forms. As such, we  consider the task of developing tools to handle the irregular case beyond the scope of this Article.


\subsection{Variational characterization of slow manifolds in nearly-periodic Hamiltonian systems}
In Section \ref{sms_for_nps} we constructed slow manifolds for nearly-periodic systems as fixed point sets for certain truncations of the roto-rate. To facilitate the study of normal stability in nearly-periodic Hamiltonian systems, this subsection will enhance that construction through the use of the adiabatic invariant $\mu_\epsilon$ described in Definition \ref{mu_defined}. In particular, for nearly-periodic Hamiltonian systems on regular, exact, barely-symplectic manifolds, we will construct slow manifolds that coincide with the set of critical points for certain truncations of $\mu_\epsilon$. These manifolds will also comprise fixed-point sets of truncations of the roto-rate, and thereby represent special cases of the manifolds given by Theorem\,\ref{sm_existence_thm}.


Our construction will make use of the well-known result \cite{Guillemin_Sternberg_1982} from symplectic geometry that any $U(1)$-momentum map on a symplectic manifold is \emph{Morse-Bott} with critical manifold equal to the set of fixed points for the underlying $U(1)$-action. For us, certain truncations of the adiabatic invariant will serve as the momentum map, and the corresponding critical manifold will provide us with our desired slow manifold. However, since regular barely-symplectic manifolds are not the same as ordinary symplectic manifolds, we will need to resort to the heuristic mentioned in Section \ref{bsm_section} in our analysis. Namely, if $\Omega_\epsilon$ is a regular barely-symplectic form with degeneracy index $d$, we should include at least the first $d+1$ terms in the $2$-form's $\epsilon$-power series before proceeding with symplectic methods.

We begin with a precise definition of Morse-Bott functions.
\begin{definition}
Let $f:M\rightarrow\mathbb{R}$ be a smooth function on a manifold $M$, and denote the set of critical points for $f$ as $S_f = \{s\in M\mid \mathbf{d}f_s = 0\}$. For each $s\in S_f$, the quadratic term in $f$'s Taylor expansion at $s$ defines a unique symmetric bilinear form $\bm{H}_s(f)\in T_s^*M\otimes  T_s^*M$ called the \textbf{Hessian form}. The function $f$ is \textbf{Morse-Bott} when
\begin{itemize}
    \item the set $S_f$ is a smooth embedded submanifold,
    \item for each $s\in S_f$ the null space of $\widehat{\bm{H}}_s(f):T_sM\rightarrow T_s^*M$ is precisely $T_sS_f$.
\end{itemize}
When $f$ is Morse-Bott, we say $S_f$ is the \textbf{critical manifold} for $f$. Since $\mathrm{im}\,\widehat{\bm{H}}_s(f) = (T_sS_f)_0$, the space of $1$-forms that annihilate $T_sS_f$, $\widehat{H}_s(f)$ induces a linear isomorphism $\widehat{\bm{H}}_s^\perp(f):T_sM/T_sS_f\rightarrow (T_sS_f)_0$ called the \textbf{transverse Hessian operator}.
\end{definition}

Next we explain an important mechanism by which Morse-Bott functions arise in nature. In particular, we will show that Noether invariants are automatically Morse-Bott. This result is striking since generic smooth functions are not Morse-Bott.

\begin{proposition}[Gillemin and Sternberg \cite{Guillemin_Sternberg_1982}]\label{symplectic_morse_bott}
If $(M,\Omega)$ is a symplectic manifold equipped with a symplectic $U(1)$-action and corresponding momentum map $\mu$, then $\mu$ is Morse-Bott.
\end{proposition}
\begin{proof}
Let $\Phi_\theta:M\rightarrow M$ denote symplectic $U(1)$-action, and let $R$ denote the corresponding  infinitesimal generator. By the definition of momentum maps, we have $\iota_R\Omega = \mathbf{d}\mu$. It follows from non-degeneracy of $\Omega$ that the set of critical points $S_\mu$ for $\mu$ coincides the the set of fixed points for $R$. Since the set of fixed points for any circle action is a smooth submanifold by Lemma \ref{S0_is_submanifold}, we see that $S_\mu\subset M$ is a smooth submanifold. 

Now consider a point $s\in S_\mu$ and the corresponding Hessian form $\bm{H}_s(\mu)$. Let $\widehat{r}:T_sM\rightarrow T_sM$ be the infinitesimal generator of the linearized $U(1)$-action $T_s\Phi_\theta:T_sM\rightarrow T_sM$. Taylor expanding $\iota_R\Omega = \mathbf{d}\mu$ to first order at $s$ implies $\Omega_s(\widehat{r}_s\,X_1,X_2) = \bm{H}_s(\mu)(X_1,X_2)$ for each pair $X_1,X_2\in T_sM$. In terms of the bundle maps corresponding to $\Omega$ and $\bm{H}(\mu)$, the last condition is equivalent to 
\begin{align}
    \widehat{\Omega}_s\,\widehat{r}_s= \widehat{\bm{H}}_s(\mu).\label{Hessian_formula_symplectic}
\end{align}
By non-degeneracy of $\Omega_s$, \eqref{Hessian_formula_symplectic} shows $\widehat{\bm{H}}_s(\mu)\,X = 0$ if and only if $\widehat{r}_s\,X = 0$. This completes the proof since $\text{ker}\,\widehat{r}_s = T_sS_\mu$.
\end{proof}

To connect the preceding results with nearly-periodic Hamiltonian systems on barely-symplectic manifolds, we must reckon with the fact that the $\epsilon\rightarrow 0$ limit of a barely-symplectic form $\Omega_\epsilon$ may be degenerate. This implies that Proposition \ref{symplectic_morse_bott} must be applied with care, since non-degeneracy of $\Omega$ is crucial to that result. We are therefore motivated to introduce the notion of a \emph{barely-Morse-Bott} function.

\begin{definition}
A smooth $\epsilon$-dependent function $f_\epsilon:M\rightarrow\mathbb{R}$ is \textbf{barely-Morse-Bott} if there is an $\epsilon_0>0$, a manifold $\Sigma_0$, and a smooth $\epsilon$-dependent embedding $\mathcal{S}_\epsilon:\Sigma_0\rightarrow M$, $\epsilon\in[-\epsilon_0,\epsilon_0]$, such that
\begin{itemize}
    \item $f_\epsilon$ is Morse-Bott for each non-zero $\epsilon\in[-\epsilon_0,\epsilon_0]$
    \item $\mathcal{S}_\epsilon(\Sigma_0) $ is the critical manifold for $f_\epsilon$ for each non-zero $\epsilon\in [-\epsilon_0,\epsilon_0]$.
\end{itemize}
We say $\mathcal{S}_\epsilon$ is the \textbf{critical embedding} for $f_\epsilon$. A barely-Morse-Bott function $f_\epsilon$ is \textbf{regular} if for each non-zero $\epsilon\in[-\epsilon_0,\epsilon_0]$ and critical point $s\in S_{f_\epsilon}$ we have $[\widehat{\bm{H}}_s^\perp(f_\epsilon)]^{-1} = \epsilon^{-\ell}\,\widehat{\bm{L}}_s(f_\epsilon)$ for some non-negative integer $\ell$ and smooth $\epsilon$-dependent bunlde map $\widehat{\bm{L}}_s(f_\epsilon):(T_sS_{f_\epsilon})_0\rightarrow T_sM/T_sS_{f_\epsilon}$. The smallest such $\ell$ is the \textbf{degeneracy index} of $f_\epsilon$.
\end{definition}

As a final preparatory step, we introduce some useful terminology for discussing nearly-periodic Hamiltonian systems with adiabatic invariants $\mu_\epsilon$ that vanish at the first few orders in $\epsilon$. While this vanishing phenomenon may seem like a technical curiosity, it occurs in important applications such as magnetized charged particle dynamics, and also may be exploited to strengthen our eventual results on normal stability.

\begin{definition}
If $X_\epsilon$ is a nearly-periodic Hamiltonian system with adiabatic invariant $\mu_\epsilon = \mu_0 + \epsilon\,\mu_1 + \dots$, the \textbf{vanishing index} for $\mu_\epsilon$ is the largest non-negative integer $\nu$ such that $\mu_\epsilon = O(\epsilon^v)$. If $\mu_\epsilon$ has vanishing index $\nu$, so that $\mu_k = 0$ for $k<v$, the associated \textbf{reduced adiabatic invariant} $\mu_\epsilon^* = \mu_0^* + \epsilon\,\mu_1^* + \epsilon^2\,\mu_2^* + \dots$ is the unique formal power series such that
\begin{align}
    \mu_\epsilon = \epsilon^{v}\,\mu_\epsilon^*.\label{mu_star_def}
\end{align}

\end{definition}



\begin{theorem}\label{crit_mu_is_SM}
Let $X_\epsilon$ be a nearly-periodic Hamiltonian system on the barely-symplectic manifold $(M,\Omega_\epsilon)$. Assume $\Omega_\epsilon$ is exact and regular with degeneracy index $d$. Also assume that the adiabatic invariant $\mu_\epsilon = \epsilon^v\,\mu_\epsilon^*$ has vanishing index $\nu\geq 0$. For each integer $N\geq 0$ and compact, codimension-$0$ submanifold $C_0\subset M$, with or without boundary, there exists an increasing sequence of codimension-$0$ compact submanifolds, $C_0\subset C_0^\prime\subset C_0^{\prime\prime}$, with $\mathrm{int}\,C_0^\prime\supset C_0$, $\mathrm{int}\,C_0^{\prime\prime}\supset C_0^\prime$, and a 
smooth $\epsilon$-dependent function $\mu_\epsilon^{*(N)}:\mathrm{int}\,C_0^{\prime\prime}\rightarrow \mathbb{R}$ such that
\begin{itemize}
    \item[(1)] $\mu_\epsilon^{*(N)} - \mu_\epsilon^* = O(\epsilon^{N+1})$ on $\mathrm{int}\,C_0^\prime$
    \item[(2)] $\mu_\epsilon^{*(N)}$ is barely-Morse-Bott with critical embedding $\mathcal{S}_\epsilon^{(N)}:S_0\cap \mathrm{int}\,C_0^{\prime\prime}\rightarrow \mathrm{int}\,C_0^{\prime\prime} $.
    \item[(3)] $\mu_\epsilon^{*(N)}$ is regular with degeneracy index at most $d-\nu$.
    \item[(4)] $\mathcal{S}_\epsilon^{(N)}\mid S_0\cap \mathrm{int}\,C_0$ is an $N^{\mathrm{th}}$-order parameterized slow manifold for $X_\epsilon$ whose image is contained in $\mathrm{int}\,C_0^\prime$ for small enough $\epsilon$.
\end{itemize}

    

\end{theorem}

\begin{remark}
The Theorem \underline{does not say} $\mu_\epsilon^* - \mu_\epsilon^{*(N)} = O(\epsilon^{N+1})$ on $\mathrm{int}\,C_0^{\prime\prime}$. 
\end{remark}

\begin{proof}
The proof begins by repeating the argument from the proof of Theorem \ref{sm_existence_thm}, but using the generator $K_\epsilon$ provided by Proposition \ref{KL_generator} in place of the generator $G_\epsilon$ provided by Proposition \ref{existence_of_normalizing_transformation}. Recall that this argument constructs an increasing sequence of compact codimension-$0$ submanifolds $C_0\subset C_0^\prime\subset C_0^{\prime\prime}$ with the desired nesting property when $\partial C_0\neq \emptyset$; when $\partial C_0 = \emptyset$ we now take $C_0^{\prime\prime}=C_0^\prime  = C_0$. In this manner, for each  integer $N\geq 0$, we obtain a smooth $\epsilon$-dependent diffeomorphism $\Psi_\epsilon:M\rightarrow M$ that is equal to the exponential of the vector field $\mathcal{K}_\epsilon^{(N)} = w\,K_\epsilon^{(N)}$. Here $w=1$ in $C_0^\prime$ and $w=0$ outside of $C_0^{\prime\prime}$, which ensures $\Psi_\epsilon(C_0^{\prime\prime}) = C_0^{\prime\prime}$. We also know that $\mathcal{S}_\epsilon^{(N)} = \Psi_\epsilon^{-1}|S_0\cap \text{int}\,C_0^{\prime\prime}$ gives a smooth $\epsilon$-dependent embedding that restricts to an $N^{\text{th}}$-order parameterized slow manifold on $S_0\cap \text{int}\,C_0$.

Next we construct the function $\mu_\epsilon^{*(N)}$. Let $d$ denote the degeneracy index for $\Omega_\epsilon$ and let $\overline{\Theta}_k$, $k\in 0,\dots,d$, be the $1$-forms given by Proposition \ref{KL_generator}. We introduce the smooth $\epsilon$-dependent function $\overline{\mu}_\epsilon = \iota_{R_0}(\overline{\Theta}_0 + \epsilon\,\overline{\Theta}_1 + \dots + \epsilon^d\,\overline{\Theta}_d)$. By construction, this function satisfies the Hamilton equation $\mathbf{d}\overline{\mu}_\epsilon = \iota_{R_0}\overline{\Omega}_\epsilon$, where $\overline{\Omega}_\epsilon$ is defined in the statement of Proposition \ref{KL_generator}. Since $\exp(\mathcal{L}_{K_\epsilon})\overline{\Omega}_\epsilon = \Omega_\epsilon$ and $\exp(\mathcal{L}_{K_\epsilon})R_0 = R_\epsilon$, we also have the formal power series identity $\iota_{R_\epsilon}\Omega_\epsilon = \mathbf{d}[\exp(\mathcal{L}_{K_\epsilon})\overline{\mu}_\epsilon]$. But because the same identity is satisfied with the adiabatic invariant $\mu_\epsilon$ in place of $\exp(\mathcal{L}_{K_\epsilon})\overline{\mu}_\epsilon$, we must have $\exp(\mathcal{L}_{K_\epsilon})\overline{\mu}_\epsilon = \mu_\epsilon + c_\epsilon$, where $c_\epsilon$ is some formal power series with constant coefficients. Using the fact that $\mu_\epsilon$ and $\exp(\mathcal{L}_{K_\epsilon})\overline{\mu}_\epsilon$ each vanish on the zero locus $R_\epsilon = 0$, we find $c_\epsilon = 0$, whence it follows $\exp(\mathcal{L}_{K_\epsilon})\overline{\mu}_\epsilon = \mu_\epsilon$. If $\nu$ denotes the vanishing index for $\mu_\epsilon$, we therefore obtain $\overline{\mu}_\epsilon = \epsilon^{\nu}\exp(-\mathcal{L}_{K_\epsilon})\mu_\epsilon^*$, which can only be satisfied if $\iota_{R_0}\overline{\Theta}_k =0$ for $k=0,\dots,\nu-1$, or $\overline{\mu}_\epsilon = \epsilon^\nu\,\iota_{R_0}(\overline{\Theta}_v + \dots + \epsilon^d\,\overline{\Theta}_d) = \epsilon^{\nu}(\overline{\mu}_0^* + \epsilon\,\overline{\mu}_1^* + \dots + \epsilon^{d-\nu}\,\overline{\mu}_{d-\nu}^*) $
where $\overline{\mu}_k^* = \overline{\mu}_{\nu + k}$, $k = 0,\dots,d-\nu$. Finally, we define
\begin{align}
   \mu_\epsilon^{*(N)} = \Psi_\epsilon^*(\overline{\mu}_0^* + \epsilon\,\overline{\mu}_1^* + \dots + \epsilon^{d-\nu}\,\overline{\mu}_{d-\nu}^*).\label{mu_approximation}
\end{align}

Now we would like to show that $\mu_\epsilon^{*(N)}$ defined in \eqref{mu_approximation} is a regular barely-Morse-Bott function with degeneracy index at most $d-\nu$ and critical embedding $\mathcal{S}_\epsilon^{(N)}$. We will argue by showing that $\overline{\mu}_\epsilon^* = \Phi_{\epsilon_*}\mu_\epsilon^{*(N)}$ is a regular barely-Morse-Bott function on $\text{int}\, C_0^{\prime\prime}$. First observe that because $\Omega_\epsilon$ is a regular barely-symplectic form with degeneracy index $d$, the $2$-form $\overline{\Omega}_\epsilon$ has a formal inverse of the form $\epsilon^{-d}\overline{j}_\epsilon$, where $\overline{j}_\epsilon$ is a formal power series with bivector coefficients. Lemma \ref{Darboux_lemma} therefore implies that there is an $\epsilon_0>0$ such that $\overline{\Omega}_\epsilon$ is symplectic when restricted to $ C_0^{\prime\prime}$ for all non-zero $\epsilon\in[-\epsilon_0,\epsilon_0]$. Moreover, there is a smooth $\epsilon$-dependent bivector $\overline{\kappa}_\epsilon$ such that $(\widehat{\overline{\Omega}}_\epsilon)^{-1} = -\epsilon^{-d}\,\widehat{\overline{\kappa}}_\epsilon$ for non-zero $\epsilon\in[-\epsilon_0,\epsilon_0]$. Because $\mathbf{d}\overline{\mu}_\epsilon = \iota_{R_0}\overline{\Omega}_\epsilon$, Lemma \ref{symplectic_morse_bott} implies $\overline{\mu}_\epsilon\mid  C_0^{\prime\prime} $, and therefore $\overline{\mu}_\epsilon^*\mid  C_0^{\prime\prime} = \epsilon^{-\nu}\mu_\epsilon\mid  C_0^{\prime\prime}$ is Morse-Bott for non-zero $\epsilon\in[-\epsilon_0,\epsilon_0]$ with critical manifold $S_0\cap \text{int}\, C_0^{\prime\prime}$. This shows that $\overline{\mu}_\epsilon^*\mid \text{int}\, C_0^{\prime\prime}$ is barely-Morse-Bott with critical embedding $s_0\in S_0\cap \text{int}\, C_0^{\prime\prime}\mapsto s_0\in \text{int}\, C_0^{\prime\prime}$. To see that $\overline{\mu}_\epsilon^*$ is regular, first note that for each $s\in S_0\cap \text{int}\, C_0^{\prime\prime}$ we have $\widehat{\overline{\Omega}}_{\epsilon s}\,\widehat{r}_s = \epsilon^{\nu}\,\widehat{\bm{H}}_s(\overline{\mu}_\epsilon^*)$, where $\widehat{r}_s$ denotes the linearization of $R_0$ at $s$. Then observe that because $\text{ker}\,\widehat{r}_s = T_sS_0$ the map $\widehat{r}_s$ induces a linear isomorphism $\widehat{r}_s^\perp:T_sM/T_sS_0\rightarrow \widehat{\overline{\Omega}}_{\epsilon s}^{-1}(T_sS_0)_0$, where $(T_sS_0)_0\subset T_s^*M$ comprises covectors at $s$ that annihilate tangent vectors in $T_sS_0$. Indeed, if $u\in\text{im}\,\widehat{r}_s$ then $u = \widehat{r}_s\,w$ for some $w\in T_sM$, which implies $u = \widehat{\overline{\Omega}}_{\epsilon s}^{-1}\,\widehat{\overline{\Omega}}_{\epsilon s}\,\widehat{r}_s w=\epsilon^\nu \widehat{\overline{\Omega}}_{\epsilon s}^{-1} \widehat{\bm{H}}_s(\overline{\mu_\epsilon^*}) w\in \widehat{\overline{\Omega}}_{\epsilon s}^{-1}(T_sS_0)_0 $, and therefore $\text{im}\,\widehat{r}_s = \widehat{\overline{\Omega}}_{\epsilon s}^{-1}(T_sS_0)_0$ by a dimension count. It follows that the transverse Hessian operator is given by $\widehat{\bm{H}}_s^\perp(\overline{\mu}_\epsilon^*) = \epsilon^{-v}\widehat{\overline{\Omega}}_{\epsilon s}\,\widehat{r}_s^\perp$, whose inverse is given by
\begin{align*}
    (\widehat{\bm{H}}_s^\perp(\overline{\mu}_\epsilon^*))^{-1}(\alpha) &= \epsilon^{\nu}(\widehat{r}_s^\perp)^{-1}(\widehat{\overline{\Omega}}_{\epsilon\,s})^{-1}(\alpha)\nonumber\\
    & = -\epsilon^{\nu-d}(\widehat{r}_s^\perp)^{-1}\widehat{\overline{\kappa}}_{\epsilon s}(\alpha),
\end{align*}
for each $\alpha\in (T_sS_0)_0$. We conclude that $\overline{\mu}_\epsilon^*$ is regular with degeneracy index $\ell \leq  d - \nu$.

To complete the proof, we now recall that our previous remarks imply $\exp(\mathcal{L}_{K_\epsilon})\overline{\mu}_\epsilon^* = \mu_\epsilon^*$ in the sense of formal power series. This implies that $\mu_\epsilon^{*(N)} = \Psi_\epsilon^*\overline{\mu}_\epsilon^*$ agrees with $\mu_\epsilon^*$ within $O(\epsilon^{N+1})$ on $\text{int}\,C_0^\prime$ because so does $\mathcal{K}_\epsilon = w\,K_\epsilon^{(N)}$ agree with $K_\epsilon$ on $\text{int}\,C_0^\prime$. 

\end{proof}

\subsection{Free-action stability principle}
We now find ourselves in a good position to prove the free-action principle for the slow manifolds provided by Theorem \ref{crit_mu_is_SM}. In our proof, we will bound the distance between a trajectory and a normally-elliptic slow manifold using adiabatic invariance and the quadratic approximation of the adiabatic invariant along the slow manifold. To that end, we will need a pair of technical lemmas.

\begin{lemma}\label{pd_bound_lemma}
Let $(E,g)$ be a real inner-product space with inner product $g$. Let $D_\epsilon:E\rightarrow E$ be a smooth $\epsilon$-dependent linear map. Suppose there exists a positive real number $\epsilon_0$ such that $D_\epsilon$ is positive definite for all $\epsilon\in (0,\epsilon_0]$. Then for all $e\in E$ and  $\epsilon\in(0,\epsilon_0]$
\begin{align*}
    g(e,e)\leq \| [D_\epsilon]^{-1}\| g(e,D_\epsilon e),
\end{align*}
where $\|\cdot\|$ denotes the induced operator norm.
\end{lemma}

\begin{proof}
Let $\lambda(A)$ and $\Lambda(A)$ denote the smallest and largest eigenvalues of a linear map $A:E\rightarrow E$. Define the induced operator norm $\| A\| = \text{sup}_{\| e\|=1}\|Ae\|$. Recall that whenever $A$ is symmetric positive-definite we have have $\|A\| = \Lambda(A)$.

Since $D_\epsilon$ is symmetric positive-definite for $\epsilon\in(0,\epsilon_0]$, we have the simple inequality
\begin{align*}
    g(e,D_\epsilon e) \geq \lambda(D_\epsilon)\,g(e,e).
\end{align*}
for all $e\in E$ and $\epsilon\in (0,\epsilon_0]$.
Since $\lambda(D_\epsilon) = 1/\Lambda([D_\epsilon]^{-1})$ and $\Lambda([D_\epsilon]^{-1}) = \| [D_\epsilon]^{-1}\|$, the desired result follows.

\end{proof}


\begin{lemma}\label{adiabatic_invariance_lemma}
Let $X_\epsilon$ be a $C^\infty$ nearly-periodic Hamiltonian system on $M$ with reduced adiabatic invariant $\mu_\epsilon^{*}$. Fix an $\epsilon_0>0$, a compact set $C\subset M$, a non-negative integer $N$, and a smooth $\epsilon$-dependent function $\mu_\epsilon^{*(N)}$ with $\mu_\epsilon^{*} - \mu_\epsilon^{*(N)} = O(\epsilon^{N+1})$ on $C$. For each non-negative integer $k$ there is a  $k$-dependent constant $\chi_k>0$ such that 
\begin{align}
    \forall\,t\in[-\epsilon^{-k},\epsilon^{-k}],\,\quad|\mu_\epsilon^{*(N)}(z(t)) - \mu_\epsilon^{*(N)}(z(0))|\leq \epsilon^{N+1}\,\chi_k \label{mu_estimate_formula}
\end{align}
for all $X_\epsilon$-integral curves $z:\mathbb{R}\rightarrow M$ contained in $C$ and all $\epsilon\in[-\epsilon_0,\epsilon_0]$.
\end{lemma}
\begin{proof}
The result follows from two basic estimates. 
\\ \\
\noindent\underline{First estimate} -- For each $n\geq N$ let $\mu_{\epsilon}^{*(n)}$ be a smooth, $\epsilon$-dependent, $n^{\text{th}}$-order approximation of $\mu_\epsilon^*$. (Such functions may be constructed by merely truncating the formal power series $\mu_\epsilon^*$ at the appropriate order.) By all-orders invariance of $\mu_\epsilon^*$, the smooth $\epsilon$-dependent function $\mathcal{L}_{X_\epsilon}\mu_\epsilon^{*(n)}$ has the formal power series expansion
\begin{align*}
    \mathcal{L}_{X_\epsilon}\mu_\epsilon^{*(n)}  &= \mathcal{L}_{X_\epsilon}(\mu_\epsilon^{*(n)} - \mu_\epsilon^{*}) = O(\epsilon^{n+1}).  
\end{align*}
Taylor's theorem with remainder therefore implies the existence of a smooth $\epsilon$-dependent function $f_\epsilon^{(n)}$ such that $\mathcal{L}_{X_\epsilon}\mu_\epsilon^{*(n)} = \epsilon^{n+1}\,f_\epsilon^{(n)}$. Thus, if $z:\mathbb{R}\rightarrow M$ is any $X_\epsilon$-integral curve, we have
\begin{align}
    \mu_\epsilon^{*(n)}(z(t)) - \mu_\epsilon^{*(n)}(z(0)) = \epsilon\,^{n+1}\,\int_0^tf_\epsilon^{(n)}(z(\overline{t}))\,d\overline{t}\label{mu_integral_equation} 
\end{align}
for each $t\in\mathbb{R}$. Let $F^{(n)}$ denote the maximum value of the continuous function $(\epsilon,z)\mapsto |f_\epsilon^{(n)}(z)|$ on the compact set $[-\epsilon_0,\epsilon_0]\times C$. The formula \eqref{mu_integral_equation} implies in particular
\begin{align}
    |\mu_\epsilon^{*(n)}(z(t)) - \mu_\epsilon^{*(n)}(z(0))|\leq \epsilon^{n+1}\,|t|\,F^{(n)},\label{estimate_one}
\end{align}
for each $X_\epsilon$-integral curve $z$ contained in $C$ and each $t\in \mathbb{R}$. Equation \eqref{estimate_one} provides our first important estimate.
\\ \\
\noindent\underline{Second estimate} -- Because $n\geq N$, there must be a smooth $\epsilon$-dependent function $\Delta \mu_\epsilon^{(n,N)}$ such that $\mu_\epsilon^{*(n)} - \mu_\epsilon^{*(N)} = \epsilon^{N+1}\,\Delta \mu_\epsilon^{*(n,N)}$ in $C$. If $\Delta\mu^{*(n,N)}$ denotes the maximum value of $(\epsilon,z)\mapsto |\Delta \mu_\epsilon^{*(n,N)}(z)|$ on $[-\epsilon_0,\epsilon_0]\times C$, we therefore have the following bound on the difference,
\begin{align}
    |\mu_\epsilon^{*(n)}(z) - \mu_\epsilon^{*(N)}(z)|\leq \epsilon^{N+1}\,\Delta\mu^{*(n,N)},\label{estimate_two}
\end{align}
for $(\epsilon,z)\in[-\epsilon_0,\epsilon_0]\times C$.
\\ \\
\noindent\underline{Combining the estimates} -- Using the estimates \eqref{estimate_one} and \eqref{estimate_two}  with $n=N+k$, we now have
\begin{align*}
    |\mu_\epsilon^{*(N)}(z(t)) - \mu_\epsilon^{*(N)}(z(0))| = \bigg|&\bigg(\mu_\epsilon^{*(N)}(z(t)) - \mu_\epsilon^{*(N+k)}(z(t))\bigg)\nonumber\\
    -&\bigg(\mu_\epsilon^{*(N)}(z(0))-\mu_\epsilon^{*(N+k)}(z(0))\bigg)\nonumber\\
    +&\bigg(\mu_\epsilon^{*(N+k)}(z(t)) - \mu_\epsilon^{*(N+k)}(z(0))\bigg)\bigg|\nonumber\\
    &\hspace{-1em}\leq 2 \epsilon^{N+1}\,\Delta \mu^{*(N+k,N)}+ \epsilon^{N+1+k}\,|t|\,F^{(N+k)}\nonumber\\
    &\hspace{-1em}\leq \epsilon^{N+1}\bigg(2\Delta \mu^{*(N+k,N)} +F^{(N+k)} \bigg),
\end{align*}
for each $X_\epsilon$-integral curve $z:\mathbb{R}\rightarrow M$ contained in $C$ and  $t\in [-\epsilon^{-k},\epsilon^{-k}]$. This proves the theorem with $\chi_k = 2\Delta \mu^{*(N+k,N)} +F^{(N+k)}$.

\end{proof}

\begin{theorem}[Free-action principle]\label{free_action_thm}
Let $X_\epsilon$ be a nearly-periodic Hamiltonian system on the barely-symplectic manifold $(M,\Omega_\epsilon)$. Assume $\Omega_\epsilon$ is regular and exact, with degeneracy index $d$. Also assume the adiabatic invariant $\mu_\epsilon$ has vanishing index $\nu\geq 0$. Fix a compact codimension-$0$ submanifold, with or without boundary, $C_0\subset M$.
Let $\mathcal{S}_\epsilon^{(N)}:S_0\cap \mathrm{int}\,C_0\rightarrow M$ and $\mu_\epsilon^{*(N)}$ denote the $N^{\mathrm{th}}$-order parameterized slow manifold and approximate adiabatic invariant provided by Theorem \ref{crit_mu_is_SM}, respectively. We require $N+1>3(d-\nu)$.

Assume that, for all sufficiently-small $\epsilon$, $\bm{H}_{s}(\mu_\epsilon^*)$ is positive or negative semi-definite for all $s$ in the closure of the image of $\mathcal{S}_\epsilon^{(N)}$. There is an $\epsilon_0>0$ such that, for all non-zero $\epsilon\in[-\epsilon_0,\epsilon_0]$ and all $X_\epsilon$-integral curves $z:\mathbb{R}\rightarrow M$ contained in $C_0$ that begin within $\epsilon^{N+1}$ of $S_\epsilon^{(N)} =\mathcal{S}_\epsilon^{(N)}(S_0)$, $z$ will either (a) remain within $\epsilon^{(N+1-d+\nu)/2}$ of $S_\epsilon^{(N)}$ for $t\in[-\epsilon^{-k},\epsilon^{-k}]$ for each positive integer $k$, or (b) eventually run off the edge of $S_\epsilon^{(N)}$.

\end{theorem}

\begin{remark}
Notice that the bound on a trajectory's normal deviation becomes weaker as $d-\nu$ becomes larger. In particular, the Theorem provides no bound at all for slow manifolds with order $N\leq 3(d-\nu)-1$. This suggests that, in general, larger degeneracy indices for the barely-symplectic form have destablizing effects on the slow manifolds constructed in this article. It also suggests larger vanishing indices for the adiabatic invariant have a stablizing effect. We describe a particular way in which these effects manifest themselves in an example following the proof of the Theorem.

\end{remark}

\begin{proof}
Given a submanifold $S\subset M$, denote the normal bundle to $S$ with respect to the auxilliary Riemannian metric $g$ on $M$ as $NS$. Denote the radius-$r$ tubular neighborhood of $S$ with respect to $g$ as $\mathcal{T}_r(S) = \{m\in M\mid \text{distance}(m,S)< r\}$; the closure of $\mathcal{T}_r(S)$ as $\overline{\mathcal{T}}_r(S)$; and the radius- $\leq r$ restriction of the normal bundle as $N^rS$.

Choose $\epsilon_0$ small enough to ensure $\overline{\mathcal{T}}_{\epsilon^{N+1}}(S_\epsilon^{(N)})\subset C_0^\prime$ and ${\mathcal{T}}_{r_0}(S_\epsilon^{(N)})\approx N^{r_0}S_\epsilon^{(N)}$ by way of the Riemannian exponential map for some $O(1)$ positive constant $r_0$ and for each $\epsilon\in[-\epsilon_0,\epsilon_0]$. Note that the set $\overline{\mathcal{T}}_{\epsilon^{N+1}}(S_\epsilon^{(N)})$, being a closed subset of the compact set $C_0^\prime$, is itself compact. Also note that by way of the diffeomorphism ${\mathcal{T}}_{r_0}(S_\epsilon^{(N)})\approx N^{r_0}S_\epsilon^{(N)}$, we may identify points in ${\mathcal{T}}_{\epsilon^{N+1}}(S_\epsilon^{(N)})$ with pairs $(s,n)$, where $s\in S_\epsilon^{(N)}$ and $n\in N_sS_\epsilon^{(N)}$.

Since, for each sufficiently small $\epsilon$,  $\bm{H}_{s}(\mu_\epsilon^{*(N)})$ is sign-semi-definite for all $s\in \overline{S}_\epsilon^{(N)}$, we may shrink $\epsilon_0$ in order to ensure sign semi-definiteness for each $s\in \overline{S}_\epsilon^{(N)}$ uniformly in $\epsilon\in[-\epsilon_0,\epsilon_0]$. If we introduce the symmetric linear map $\bm{D}_{\epsilon\,s}:N_sS_\epsilon^{(N)}\rightarrow N_sS_\epsilon^{(N)}$ by requiring $\bm{H}_{s}(\mu_\epsilon^{*(N)})(n,n) = g_s(n,\bm{D}_{\epsilon\,s}n)$ for each $n\in N_sS_\epsilon^{(N)}$, the barely-Morse-Bott property for $\mu_\epsilon^{*(N)}$ therefore implies $\bm{D}_{\epsilon\,s}$ is symmetric positive- or negative-definite for each $s\in\overline{S}_\epsilon^{(N)}$ and non-zero $\epsilon\in[-\epsilon_0,\epsilon_0]$.

Because $\mu_\epsilon^{*(N)}$ is regular with degeneracy index at most $d-\nu$, Lemma \ref{pd_bound_lemma} and compactness of $\overline{S}_\epsilon^{(N)}$ implies there is a positive constant $D_0$, depending only on $S_\epsilon^{(N)}$ and $\mu_\epsilon^{*(N)}$, such that
\begin{align}
    g_s(n,n)\leq \frac{1}{\epsilon^{d-\nu}\,D_0}|g_s(n,\bm{D}_{\epsilon\,s}n)|,
\end{align}
for each $s\in\overline{S}_\epsilon^{(N)}$, $n\in N_s\overline{S}_\epsilon^{(N)}$, and non-zero $\epsilon\in [-\epsilon_0,\epsilon_0]$. By Taylor's theorem with remainder and $\mu_\epsilon^{*(N)}= 0$ on $S_\epsilon^{(N)}$, we also have the following inequality relating $\bm{D}_{\epsilon\,s}$ to $\mu_\epsilon^{*(N)}$:
\begin{align}
    |g_s(n,\bm{D}_{\epsilon\,s}n)|\leq |\mu_\epsilon^{*(N)}(s,n)| + T_0\,[g_s(n,n)]^{3/2} 
\end{align}
for each $s\in \overline{S}_\epsilon^{(N)}$, $n\in N_s\overline{S}_\epsilon^{(N)}$, and $\epsilon\in [-\epsilon_0,\epsilon_0]$. Here $T_0$ is a positive constant that depends only on $S_\epsilon^{(N)}$ and $\mu_\epsilon^{*(N)}$; it bounds the third normal derivative of $\mu_\epsilon^{*(N)}(s,n)$ on $\overline{S}_\epsilon^{(N)}$. Combining the previous two estimates, we therefore obtain the key geometric inequality
\begin{align}
    g_s(n,n)\leq  \frac{1}{\epsilon^{d-\nu}\,D_0}\bigg( |\mu_\epsilon^{*(N)}(s,n)| + T_0\,[g_s(n,n)]^{3/2}\bigg),\label{geom_inequality}
\end{align}
for each $s\in\overline{S}_\epsilon^{(N)}$, $n\in N_s\overline{S}_\epsilon^{(N)}$, and non-zero $\epsilon\in [-\epsilon_0,\epsilon_0]$.

Now suppose that $z:\mathbb{R}\rightarrow M$ is an $X_\epsilon$-integral curve contained in $C_0$ that begins in the narrow tubular neighborhood ${\mathcal{T}}_{\epsilon^{N+1}}(S_\epsilon^{(N)})\subset C_0^\prime$. Let $\mathcal{I}_0 = (a,b)$ be the maximal time interval during which $z(t)$ is contained in $\mathcal{T}_{r_0}(S_\epsilon^{(N)})$. For $t\in \mathcal{I}_0$ , we write $z(t) = (s(t),n(t))$.  By the geometric inequality \eqref{geom_inequality} we have
\begin{align}
    g_{s(t)}(n(t),n(t))\leq  \frac{1}{\epsilon^{d-\nu}\,D_0}\bigg( |\mu_\epsilon^{*(N)}(s(t),n(t))| + T_0\,[g_{s(t)}(n(t),n(t))]^{3/2}\bigg),
\end{align}
for $t\in \mathcal{I}_0$.
But by the near-constancy of $\mu_\epsilon^{*(N)}(s(t),n(t))$ given by Lemma \ref{adiabatic_invariance_lemma}, we anticipate that this inequality should allow us to bound the distance $d(t) = [g_{s(t)}(n(t),n(t))]^{1/2}$ between $S_\epsilon^{(N)}$ and and $z(t)$. The following analysis makes this intuition precise.

By Lemma \ref{adiabatic_invariance_lemma}, for each non-negative integer $k$ there is a non-negative constant $\chi_k$ such that the inequality \eqref{mu_estimate_formula} is satisfied for any $X_\epsilon$-integral curve contained in $C_0$. The inequality holds in particular for $z(t)$ introduced in the previous paragraph for $t\in \overline{\mathcal{I}}_k\equiv [-\epsilon^{-k},\epsilon^{k}]\cap \mathcal{I}_0$, giving
\begin{align}
    g_{s(t)}(n(t),n(t))\leq  \frac{1}{\epsilon^{d-\nu}\,D_0}\bigg( \epsilon^{N+1}\,\chi_k +|\mu_\epsilon^{*(N)}(s(0)),n(0))| + T_0\,[g_{s(t)}(n(t),n(t))]^{3/2}\bigg),\label{vv1}
\end{align}
for $t\in \overline{\mathcal{I}}_k$. Note that if we introduce the polynomial 
\begin{align}
    P_\epsilon(d) = d^2 - \frac{T_0}{\epsilon^{d-\nu}\,D_0}\,d^3,
\end{align}
we may write \eqref{vv1} equivalently as
\begin{align}
    P_\epsilon(d(t)) \leq  \frac{1}{\epsilon^{d-\nu}\,D_0}\bigg( \epsilon^{N+1}\,\chi_k +|\mu_\epsilon^{*(N)}(s(0)),n(0))| \bigg).\label{polynomial_form}
\end{align}
Again using Taylor's theorem with remainder and $\mu_\epsilon^{*(N)}= 0$ on $S_\epsilon^{(N)}$, we estimate the size of the initial reduced adiabatic invariant according to
\begin{align}
    |\mu_\epsilon^{*(N)}(s(0)),n(0))|&\leq |g_{s(0)}(n(0),\bm{D}_{{\epsilon\,s(0)}}\,n(0))| + T_0\,[g_{s(0)}(n(0),n(0))]^{3/2}\nonumber\\
    &\leq D_1\,d^2(0) + T_0\,d^3(0)\nonumber\\
    &\leq \epsilon^{2(N+1)}\,D_1 + \epsilon^{3(N+1)}\,T_0,
\end{align}
where $D_1$ is uniform bound on $\bm{D}_{\epsilon\,s}$ for $(\epsilon,s)\in[-\epsilon_0,\epsilon_0]\times \overline{S}_\epsilon^{(N)}$. The inequality \eqref{polynomial_form} then becomes
\begin{align}
    P_\epsilon(d(t)) \leq  \frac{1}{\epsilon^{d-\nu}\,D_0}\bigg( \epsilon^{N+1}\,\chi_k +\epsilon^{2(N+1)}\,D_1 + \epsilon^{3(N+1)}\,T_0 \bigg),\label{vv2}
\end{align}
for $t\in\overline{\mathcal{I}}_k$. Now, for $d\geq 0$, $P_\epsilon$ increases monotonically from $0$ before reaching its maximum value of $P_{\text{max}} = (4/27)\,\epsilon^{2(d-\nu)}\,(D_0/T_0)^2$ at $d_{\text{max}} = (2/3)\,\epsilon^{d-\nu}\,(D_0/T_0)$. Thus, if $N+1 - d + \nu >2(d-\nu)$ and we shrink $\epsilon_0$, if necessary, then $d(0) <\epsilon^{N+1} <d_{\text{max}}$ and $P(d(t)) < P_{\text{max}}$ for $t\in \overline{\mathcal{I}}_k$, by \eqref{vv2}. It follows that the distance $d(t)$ is bounded by the smallest non-negative solution $d^*$ of the polynomial equation
\begin{align}
    P_\epsilon(d^*) = \frac{1}{\epsilon^{d-\nu}\,D_0}\bigg( \epsilon^{N+1}\,\chi_k +\epsilon^{2(N+1)}\,D_1 + \epsilon^{3(N+1)}\,T_0 \bigg),\label{find_the_bound}
\end{align}
for $t\in \overline{\mathcal{I}}_k$. Since $d^*\sim \epsilon^{(N+1-d+v)/2}\sqrt{\chi_k/D_0}$ as $\epsilon\rightarrow 0$, we can shrink $\epsilon_0$ further to ensure
\begin{align}
    d(t)\leq \epsilon^{(N+1-d+v)/2}\sqrt{\frac{\chi_k}{D_0}},\label{the_final_bound}
\end{align}
for $t\in \overline{\mathcal{I}}_k$. We conclude that $z(t)$ either remains in the narrow tubular neighborhood $\mathcal{T}_{d^*}(S_\epsilon^{(N)})$ for $t\in[-\epsilon^{-k},\epsilon^{k}]$ or leaves the $O(1)$ tubular neighborhood $\mathcal{T}_{r_0}(S_\epsilon^{(N)})$ at some $t\in[-\epsilon^{-k},\epsilon^{k}]$. The latter possibility corresponds to $z(t)$ running ``off the edge" of the slow manifold $S_\epsilon^{(N)}$.

\end{proof}
\begin{example}
Given a positive integer $n$ and an exact symplectic manifold $(W,\omega)$, consider the product manifold $M = \mathbb{C}^{n}\times W$. Denote points $m\in M$ as $m = (z_1,\dots,z_n,w)$, with $z_k=(x_k,y_k)\in\mathbb{C}$ and $w\in W$. Equip $M$ with the regular, barely-symplectic form
\begin{align}
    \Omega_\epsilon = dx_1\wedge dy_1 + \epsilon\,dx_2\wedge dy_2 + \dots + \epsilon^{n-1}\,dx_n\wedge dy_n + \epsilon^{n-1}\,\omega.\nonumber
\end{align}
Note that the degeneracy index for $\Omega_\epsilon$ is $d=n-1$. Define the smooth $\epsilon$-dependent function
\begin{align}
    H_\epsilon(z_1,\dots,z_n,w) = \sum_{k=1}^{n}\epsilon^{k-1}\,\frac{1}{2}\,z_k\overline{z}_k + \epsilon^{n}\,U(\gamma_1,\gamma_2,\dots,\gamma_{n-1},w)
\end{align}
where $\gamma_k = z_k\overline{z}_{k+1}$, and $U:\mathbb{C}^{n-1}\times W\rightarrow \mathbb{R}$ is any smooth function. 

Consider the Hamiltonian system $X_\epsilon$ on $(M,\Omega_\epsilon)$ determined by the Hamilton equation $\iota_{X_\epsilon}\Omega_\epsilon = \mathbf{d}H_\epsilon$. Explicitly, $X_\epsilon$ is given by
\begin{align*}
    X_\epsilon &= -i \left(z_1+\epsilon^{n}\,z_2\frac{\delta U}{\delta \gamma_1}\right)\,\partial_{z_1}\nonumber\\
    & -i \left(z_2+\epsilon^{n-1}\,z_3\frac{\delta U}{\delta \gamma_2}+\epsilon^{n-1}\,\overline{z}_1\,\frac{\delta U}{\delta \gamma_1}\right)\,\partial_{z_2}\nonumber\\
    &-\dots\\
    &-i \left(z_n+\epsilon\,\overline{z}_{n-1}\frac{\delta U}{\delta \gamma_{n-1}}\right)\,\partial_{z_n} \nonumber\\
    & + \epsilon\,\widehat{\omega}^{-1}\,\mathbf{d}_wU.
\end{align*}
It is simple to verify that $X_\epsilon$ is a nearly-periodic system, whose limiting dynamics corresponds to a collection of decoupled oscillators $z_k$ with unit angular frequency. In contrast to general nearly-periodic systems, the roto-rate $R_\epsilon$ for this system is equal to the limiting roto-rate to all-orders, $R_\epsilon = R_0 = \sum_k -i\,z_k\,\partial_{z_k}$. (This is a consequence of $\mathcal{L}_{R_0}H_\epsilon = 0$ and $\mathcal{L}_{R_0}\Omega_\epsilon = 0$.) It follows that the adiabatic invariant $\mu_\epsilon$ (a true invariant in this case) is given by
\begin{align}
    \mu_\epsilon = \iota_{R_0}\sum_{k=1}^n \epsilon^{k-1}\,\frac{1}{2}(y_k\,dx_k - x_k\,dy_k) = \sum_{k=1}^n\epsilon^{k-1}\,\frac{1}{2}z_k\overline{z}_k.  
\end{align}
Note that the vanishing index for $\mu_\epsilon$ is $\nu=0$. We also know that the limiting slow manifold $S_0$ (in this case an actual invariant manifold) is an $N^{\text{th}}$-order slow manifold for each $N$.

Given a trajectory for $X_\epsilon$ that begins within $\epsilon^{N+1}$ of $S_0$, the tightest bound on the normal deviation that Theorem \ref{free_action_thm} can provide is $O(\epsilon^{(N+1-d)/2})$. Let us compare this worst-case bound with the worst-case bound implied by $\mu_\epsilon$-conservation in this example. Along our trajectory, $\mu_\epsilon = O(\epsilon^{2(N+1)})$. But since the the various oscillators $z_k$ may exchange action $\tfrac{1}{2}z_k\overline{z}_k$ by way of the interaction potential $U$ while keeping the sum $\mu_\epsilon$ fixed, the oscillator configuration with $z_k=0$ for $k<n$ and $z_n=O(\epsilon^{(2[N+1] - d)/2})$ is consistent with $\mu_\epsilon$ and $H_\epsilon$ conservation. This bound is similar, although not identical to, the bound from Theorem \ref{free_action_thm}. The discrepancy is due entirely to exact invariance of $\mu_\epsilon$; if the $\mu_\epsilon^{(N)}$ in Theorem \ref{free_action_thm} were conserved exactly, then $\chi_k$ in \eqref{find_the_bound} would vanish, and the bound implied by Theorem \ref{free_action_thm} would instead be $O(\epsilon^{(2[N+1] - d)/2})$, exactly as in this example. We therefore conjecture that the bound given by Theorem \ref{free_action_thm} cannot be improved in general. (Although it certainly can be improved in specific cases.)

\end{example}

\section{Applications to slow-manifold-embedding of guiding center dynamics\label{gc_application_section}}
We will present several applications of the general theory developed in previous Sections. For these examples, it will be useful to discuss slow manifolds in the context of fast-slow systems, which we now define and explain.

\begin{definition}
An ordinary differential equation $\dot{y} = f_\epsilon(x,y)$, $\dot{x} = \epsilon\,g_\epsilon(x,y)$, where $f_\epsilon(x,y),g_\epsilon(x,y)$ are smooth in $(\epsilon,x,y)$, is a \textbf{fast-slow system} if
\begin{align}
D_yf_0(x,y)\text{ is invertible whenever }f_0(x,y)= 0.
\end{align}
We refer to $x$ as the \emph{slow variable} and $y$ as the \emph{fast variable}.
\end{definition}

\begin{definition}\label{def:fast-slow}
A fast-slow system $\dot{y} = f_\epsilon(x,y)$, $\dot{x} = \epsilon\,g_\epsilon(x,y)$ admits a \textbf{formal slow manifold} if there is a formal power series $y_\epsilon^*(x) = y_0^*(x) + \epsilon\,y_1^*(x)+\epsilon^2\,y_2^*(x) + \dots$ that satisfies the first-order system of nonlinear partial differential equations
\begin{align}
\epsilon\,Dy_\epsilon^*(x)[g_\epsilon(x,y_\epsilon^*(x))] = f_\epsilon(x,y_\epsilon^*(x)),
\end{align}
to all orders in $\epsilon$.
\end{definition}

\begin{proposition}\label{prop:slow-manifold-series}
Eeach fast-slow system admits a unique formal slow manifold $y_\epsilon^*(x) = y_0^*(x) + \epsilon\,y_1^*(x)+\epsilon^2\,y_2^*(x) + \dots$. The first two coefficients of $y_\epsilon^*$ are determined by
\begin{align}
\label{eq:fast-slow-y0}
&f_0(x,y_0^*(x)) = 0\\
\label{eq:fast-slow-y1}
& Dy_0^*(x)[g_0(x,y_0^*(x))] = D_yf_0(x,y_0^*(x))[y_1^*(x)] + f_1(x,y_0^*(x)).
\end{align}
\end{proposition}

These results imply that fast-slow systems admit slow manifolds of each order, as defined in Definition \ref{parameterized_sm_def}. The reason these systems are so convenient is that their slow manifolds may be computed without resorting to near-identity coordinate transformations. We will use this feature of fast-slow systems to simplify computations in what follows.

\subsection{The classical Pauli particle embedding}\label{eq:classical_pauli}

As a first application, we consider the slow manifold embedding of guiding center dynamics introduced by Xiao and Qin in \cite{Xiao_2020}. We will establish long-term normal stability of this embedding in continuous time.

With $\bm{x}\in M$ and $(\bm{x},\bm{v})\in T_{\bm{x}}M$, the classical Pauli system is described by the following ordinary differential equations
\begin{align}
    \frac{d\bm{x}}{dt}=\epsilon \bm{v}, \qquad \frac{d\bm{v}}{dt} =\omega_c \bm{v}\times\bm{b}-\epsilon \mathcal{M}\nabla |\bm{B}|.
    \label{eq:classical_pauli_system}
\end{align}
Here $\mathcal{M}=\mu_P/m\in \mathbb{R}$ is a parameter, the cyclotron frequency is $\omega_c=q|\bm{B}|/m$, the vector $\bm{B}$ is the magnetic field with $\bhat=\bm{B}/|\bm{B}|$ the corresponding unit vector, and $\epsilon$ is the ordering parameter placed to indicate the cyclotron frequency as the fastest time scale in the system.

The classical Pauli system is Hamiltonian with respect to the vector field $X_\epsilon = (d\bm{x}/dt,d\bm{v}/dt)$ on the exact, regular, barely-symplectic manifold $(TM,\Omega_\epsilon)$. The one-form $\vartheta_\epsilon$, from which the barely-symplectic form is computed as $\Omega_\epsilon = -\mathbf{d}\vartheta_\epsilon$, and the Hamiltonian $H_\epsilon$ are given by
\begin{align}
    \label{eq:classical_one_form}
    \vartheta_\epsilon &= (q\bm{A}+\epsilon m\bm{v})\cdot d\bm{x}, 
    \\
    \label{eq:classical_hamiltonian}
    H_\epsilon &= \epsilon^2m(|\bm{v}|^2/2+\mathcal{M}|\bm{B}|).
\end{align}
The degeneracy index for $\Omega_\epsilon$ is $d=2$, and it is straightforward to confirm that the Hamilton's equation $\iota_{X_\epsilon}\Omega_\epsilon = \mathbf{d}H_\epsilon$ recovers \eqref{eq:classical_pauli_system} for all values of $\epsilon$. 
The nearly periodic nature is confirmed by observing the limiting vector field to be $X_0=\omega_c\bm{v}\times\bhat\cdot\partial_{\bm{v}}$ which provides the nowhere vanishing frequency function $\omega_0=\omega_c$ and the $2\pi$-periodic vector field $R_0=\bm{v}\times\bhat\cdot\partial_{\bm{v}}$ generating the $U(1)$-action on $TM$ and satisfying $\mathcal{L}_{R_0}\omega_0=0$. The $2\pi$-periodicity can be verified by analytically solving the flow of $R_0$: 
\begin{align}
    \Phi_\theta(\bm{x},\bm{v})=\exp(\theta R_0)(\bm{x},\bm{v})
    &=(\bm{x},\bm{v}\cdot\bm{b}\bm{b}+\sin\theta\bm{v}\times\bm{b}+\cos\theta\bm{b}\times(\bm{v}\times\bm{b})).
    \label{eq:classical_pauli_angle_action}
\end{align}

Next we will show that the classical Pauli system is fast-slow in order to efficiently identify the system’s slow manifold and the corresponding induced slow dynamics.

\begin{lemma}\label{classical_frame_lemma}
There exists a smooth orthonormal right-handed triad $(\ehat_1,\ehat_2,\bhat)$ on $M$. Regarding the existence of such triad, see \cite{Burby_2012}. 
\end{lemma}

\begin{lemma}
In the coordinates $(\bm{x},u,v^1,v^2)$ on $TM$ defined by \begin{align}
\bm{v}=u\bhat+v^1\ehat_1+v^2\ehat_2,\nonumber
\end{align}
where $(\ehat_1,\ehat_2,\bhat)$ is the orthonormal triplet provided by Lemma \ref{classical_frame_lemma}, the classical Pauli system \eqref{eq:classical_pauli_system} is equivalent to 
\begin{align}
\label{eq:classical_pauli_system_x}
\frac{d\bm{x}}{dt}&=\epsilon\bm{v},\\
\label{eq:classical_pauli_system_u}
\frac{du}{dt}&=\epsilon\bm{v}\cdot\nabla\bhat\cdot\bm{v}-\epsilon \mathcal{M}\bhat\cdot\nabla |\bm{B}|,\\
\label{eq:classical_pauli_system_v1}
\frac{dv^1}{dt}&=\omega_c\,\ehat_1\cdot\bm{v}\times\bhat-\epsilon(\mathcal{M}\nabla |\bm{B}|+u\bm{v}\cdot\nabla\bhat)\cdot\ehat_1-\epsilon v^2\bm{v}\cdot\bm{R},\\
\label{eq:classical_pauli_system_v2}
\frac{dv^2}{dt}&=\omega_c\,\ehat_2\cdot\bm{v}\times\bhat-\epsilon(\mathcal{M}\nabla |\bm{B}|+u\bm{v}\cdot\nabla\bhat)\cdot\ehat_2+\epsilon v^1\bm{v}\cdot\bm{R},
\end{align}
where $\bm{R}=\nabla\ehat_2\cdot\ehat_1$ is Littlejohn's \cite{Littlejohn_1984} gyrogauge vector.
\end{lemma}

\begin{proposition}
The system of ordinary differential equations \eqref{eq:classical_pauli_system_x}-\eqref{eq:classical_pauli_system_v2} comprises a fast-slow system with slow variable $x = (\bm{x},u)$ and fast variable $y = (v^1,v^2)$. The function $f_\epsilon(x,y)=(dv^1/dt,dv^2/dt)$ is given by
\begin{align}
f_0(x,y) & = \begin{pmatrix}
\omega_c\,\ehat_1\cdot\bm{v}\times\bhat\\
\omega_c\,\ehat_2\cdot\bm{v}\times\bhat
\end{pmatrix},
\end{align}
\begin{align}
f_1(x,y) & = \begin{pmatrix}
-(\mathcal{M}\nabla |\bm{B}|+u\bm{v}\cdot\nabla\bhat)\cdot\ehat_1- v^2\bm{v}\cdot\bm{R}\\
-(\mathcal{M}\nabla |\bm{B}|+u\,\bm{v}\cdot\nabla\bhat)\cdot\ehat_2+ v^1\bm{v}\cdot\bm{R}
\end{pmatrix},
\end{align}
and the function $g_\epsilon(x,y)=(d\bm{x}/dt,du/dt)$ is given by
\begin{align}
g_0(x,y) & = \begin{pmatrix}
\bm{v}\\
\bm{v}\cdot\nabla\bhat\cdot\bm{v}-\mathcal{M}\bhat\cdot\nabla |\bm{B}|
\end{pmatrix}.
\end{align}
\end{proposition}


\begin{proposition}
The first two coefficients of the formal slow manifold $y_\epsilon^*= ((v^1)^{*}_\epsilon,(v^2)^{*}_\epsilon)$ for the fast-slow system \eqref{eq:classical_pauli_system_x}-\eqref{eq:classical_pauli_system_v2} are given by
\begin{align}
\begin{pmatrix}
(v^1)^{*}_0\\
(v^2)^{*}_0
\end{pmatrix}=
\begin{pmatrix}
0\\
0
\end{pmatrix},
\end{align}
\begin{align}
\begin{pmatrix}
(v^1)^{*}_1\\
(v^2)^{*}_1
\end{pmatrix}=
\omega_c^{-1}\begin{pmatrix}
    -(\mathcal{M}\nabla |\bm{B}|+u^2\bm{\kappa})\cdot\ehat_2\\
    (\mathcal{M}\nabla |\bm{B}|+u^2\bm{\kappa})\cdot\ehat_1
\end{pmatrix},
\end{align}
where $\bm{\kappa}=\bhat\cdot\nabla\bhat$ is the magnetic field-line curvature.
In particular, if $(\bm{v}_\perp)^*_\epsilon = (v^1)^*_\epsilon\,\ehat_1 + (v^2)^*_\epsilon\,\ehat_2$, we have that
\begin{align}
    (\bm{v}_\perp)^*_\epsilon = \epsilon\omega_c^{-1}\bhat\times(\mathcal{M}\nabla |\bm{B}|+u^2\bm{\kappa})+O(\epsilon^2).
\end{align}
\end{proposition}
Investigating then the equations of motion of the slow variable along the slow-manifold, i.e., $\dot{x}=f_{\epsilon}(x,y^\ast_\epsilon(x))$, we find that 
\begin{align}
    \frac{d\bm{x}}{dt}&=\epsilon u\bhat+\epsilon^2\omega_c^{-1}\bhat\times(\mathcal{M}\nabla |\bm{B}|+u^2\bm{\kappa})+\mathcal{O}(\epsilon^3),\\
    \frac{du}{dt}&=-\epsilon(\bhat+\epsilon u\omega_c^{-1}\bhat\times\bm{\kappa})\cdot \mathcal{M}\nabla |\bm{B}|+\mathcal{O}(\epsilon^3).
\end{align}
These equations match exactly with the standard guiding-center equations derived from Littlejohn's Lagrangian if the magnitude $|\bm{B}|$ in the cyclotron frequency $\omega_c$ is replaced by the so-called $B^\ast_\parallel=(\bm{B}(+mu/q)\nabla\times\bhat)\cdot\bhat$ and $\mu_P$ in $\mathcal{M}$ is interpreted as the magnetic moment of the guiding-center. The factor $B_\parallel^\ast$ is needed to guarantee that the slow vector field $X_\epsilon^\ast=(d\bm{x}/dt,du/dt)$ is divergence free, i.e., that $\nabla\cdot(d\bm{x}/dt)+\partial_u(du/dt)=0$. Furthermore, as explained in \cite{Burby_loops_2019}, dynamics on the slow manifold is necessarily Hamiltonian. The corresponding symplectic form $\Omega_\epsilon^\ast = -\mathbf{d}\vartheta_\epsilon^\ast$ and the Hamiltonian $H^\ast_\epsilon$ are given by pulling back $\Omega_\epsilon$ and $H_\epsilon$ along the mapping $(x,y)\mapsto (x,y_\epsilon^\ast(x))$. This first provides
\begin{align}
    \vartheta_\epsilon^\ast &= (q\bm{A}+\epsilon m u\bm{b})\cdot d\bm{x}+ \mathcal{O}(\epsilon^2),\\
    H^\ast_\epsilon &= \epsilon^2m(u^2/2+\mathcal{M}|\bm{B}|)+\mathcal{O}(\epsilon^3),
\end{align}
from which the Hamilton's equations $\iota_{X_\epsilon^\ast}\Omega_\epsilon^\ast=H^\ast_\epsilon$ would provide exactly the standard guiding-center equations, with the $B_\parallel^\ast$ corrections included. 



Finally, we will demonstrate that the slow manifold for the classical Pauli system enjoys long-term normal stability. For this, we will show that the first nontrivial term in the adiabatic invariant for this system has sign-definite second variation along the limiting slow manifold $v^1 = v^2 = 0$. Let $\mu_\epsilon = \mu_0 + \epsilon\,\mu_1 + \epsilon^2\,\mu_2 + \dots$ denote the adiabatic invariant series for the classical Pauli system. According to Eq.\,(3.14) in \cite{Burby_Squire_2020}, $\mu_0 = \iota_{R_0}\langle \vartheta_0\rangle$, where we have $R_0=\bm{v}\times\bhat\cdot\partial_{\bm{v}}$ and $\vartheta_0 = q\bm{A}\cdot d\bm{x}$, and the angle brackets denote averaging over the $U(1)$-action $\Phi_\theta$ that is generated by $R_0$. Since $\Phi_\theta$ in \eqref{eq:classical_pauli_angle_action} leaves the $\bm{x}$-position fixed and $\vartheta_0$ depends only on $\bm{x}$, we have that $\langle \vartheta_0\rangle = \vartheta_0$, and since $R_0$ has only velocity components, we conclude $\mu_0=0$. Given that $\mu_0=0$, Eq.\,(3.15) in \cite{Burby_Squire_2020} then provides $\mu_1 = \iota_{R_0}\langle \vartheta_1\rangle$, where we have $\vartheta_1 = m\bm{v}\cdot d\bm{x}$. Using again the fact that $\Phi_\theta$ leaves $\bm{x}$ fixed, the average is simple to compute, giving $\langle \vartheta_1\rangle = m\bm{v}\cdot\bhat \bhat\cdot d\bm{x}$. As there again is no $d\bm{v}$ component in $\langle\vartheta_1\rangle$, the contraction $\iota_{R_0}\langle \vartheta_1\rangle $ vanishes, giving $\mu_1 =0$. Finally using Eq.\,(3.16) in \cite{Burby_Squire_2020}, we find that $\mu_2 = \frac{1}{2}\langle \mathbf{d}\vartheta_0(\mathcal{L}_{R_0}I_0\widetilde{X}_1,I_0\widetilde{X}_1)\rangle$, where $X_1=\bm{v}\cdot\partial_{\bm{x}}-\mathcal{M}\nabla|\bm{B}|\cdot\partial_{\bm{v}}$ is the first-order term in $X_\epsilon$, $\widetilde{X}_1 =X_1 - \langle X_1\rangle $, and $I_0 = \mathcal{L}_{\omega_0\,R_0}^{-1}$. 
The pullback of $X_1$ along $R_0$ is given by
\begin{align}
    X_1^\theta=\Phi^\ast_\theta X_1 &=\bm{v}\cdot\bm{b}\bm{b}\cdot\partial_{\bm{x}}+\sin\theta\bm{v}\times\bm{b}\cdot\partial_{\bm{x}}+\cos\theta\bm{b}\times(\bm{v}\times\bm{b})\cdot\partial_{\bm{x}}\nonumber
    \\
    &+\mathcal{M}(-\bm{b}\bm{b}\cdot\nabla |\bm{B}| +\sin\theta \bm{b}\times\nabla |\bm{B}| + \cos\theta\bm{b}\times(\bm{b}\times\nabla |\bm{B}|))\cdot\partial_{\bm{v}}\nonumber
    \\
    &+\{(\bm{v}\cdot\bm{b}) (\bm{b}\times\bm{\kappa})\times\bm{v}-\frac{1}{2}[\bm{b}\times(\bm{v}_\perp\cdot\nabla\bm{b})]\times\bm{v}+\frac{1}{2} [(\bm{v}\times\bm{b})\cdot\nabla\bm{b}]\times\bm{v}\}\cdot\partial_{\bm{v}}\nonumber
    \\
    &+\cos\theta \{[\bm{b}\times(\bm{v}_\perp\cdot\nabla\bm{b})]\times\bm{v}-(\bm{v}\cdot\bm{b})(\bm{b}\times\bm{\kappa})\times\bm{v}\}\cdot\partial_{\bm{v}}\nonumber
    \\
    &+\sin\theta \{[\bm{b}\times((\bm{v}\times\bm{b})\cdot\nabla\bm{b})]\times\bm{v}+(\bm{v}\cdot\bm{b})\bm{\kappa}\times\bm{v}\}\cdot\partial_{\bm{v}}\nonumber
    \\
    &+\frac{1}{2}\sin(2\theta)\{ (\bm{v}_\perp\cdot\nabla\bm{b})\times\bm{v}-[\bm{b}\times((\bm{v}\times\bm{b})\cdot\nabla\bm{b})]\times\bm{v}\}\cdot\partial_{\bm{v}}\nonumber
    \\
    &-\frac{1}{2}\cos(2\theta) \{[\bm{b}\times(\bm{v}_\perp\cdot\nabla\bm{b})]\times\bm{v}
    +[(\bm{v}\times\bm{b})\cdot\nabla\bm{b}]\times\bm{v}\}\cdot\partial_{\bm{v}}
\end{align}
This permits us to compute the inverse 
\begin{align}
    \label{eq:adiabatic_helper1}
    \omega_cI_0\widetilde{X}_1^\theta&=-\cos\theta\bm{v}\times\bm{b}\cdot\partial_{\bm{x}}+\sin\theta\bm{b}\times(\bm{v}\times\bm{b})\cdot\partial_{\bm{x}}\nonumber
    \\
    &+\mathcal{M}(-\cos\theta \bm{b}\times\nabla |\bm{B}| + \sin\theta\bm{b}\times(\bm{b}\times\nabla |\bm{B}|))\cdot\partial_{\bm{v}}\nonumber
    \\
    &+\sin\theta \{[\bm{b}\times(\bm{v}_\perp\cdot\nabla\bm{b})]\times\bm{v}-(\bm{v}\cdot\bm{b})(\bm{b}\times\bm{\kappa})\times\bm{v}\}\cdot\partial_{\bm{v}}\nonumber
    \\
    &-\cos\theta \{[\bm{b}\times((\bm{v}\times\bm{b})\cdot\nabla\bm{b})]\times\bm{v}+(\bm{v}\cdot\bm{b})\bm{\kappa}\times\bm{v}\}\cdot\partial_{\bm{v}}\nonumber
    \\
    &-\frac{1}{8}\cos(2\theta)\{ (\bm{v}_\perp\cdot\nabla\bm{b})\times\bm{v}-[\bm{b}\times((\bm{v}\times\bm{b})\cdot\nabla\bm{b})]\times\bm{v}\}\cdot\partial_{\bm{v}}\nonumber
    \\
    &-\frac{1}{8}\sin(2\theta) \{[\bm{b}\times(\bm{v}_\perp\cdot\nabla\bm{b})]\times\bm{v}
    +[(\bm{v}\times\bm{b})\cdot\nabla\bm{b}]\times\bm{v}\}\cdot\partial_{\bm{v}}
    \nonumber
    \\
    &-(\sin\theta\bm{v}\times\bm{b}+\cos\theta\bm{b}\times(\bm{v}\times\bm{b}))\cdot\nabla\ln\omega_c \bm{v}\times\bm{b}\cdot\partial_{\bm{v}}
\end{align}
and, from this, trivially the expression
\begin{align}
    \label{eq:adiabatic_helper2}
    \omega_c\mathcal{L}_{ R_0}(I_0\widetilde{X}_1^\theta)=\omega_c\partial_\theta(I_0\widetilde{X}_1^\theta)&=\sin\theta\bm{v}\times\bm{b}\cdot\partial_{\bm{x}}+\cos\theta\bm{b}\times(\bm{v}\times\bm{b})\cdot\partial_{\bm{x}}\nonumber
    \\
    &+\mathcal{M}(\sin\theta \bm{b}\times\nabla |\bm{B}| + \cos\theta\bm{b}\times(\bm{b}\times\nabla |\bm{B}|))\cdot\partial_{\bm{v}}\nonumber
    \\
    &+\cos\theta \{[\bm{b}\times(\bm{v}_\perp\cdot\nabla\bm{b})]\times\bm{v}-(\bm{v}\cdot\bm{b})(\bm{b}\times\bm{\kappa})\times\bm{v}\}\cdot\partial_{\bm{v}}\nonumber
    \\
    &+\sin\theta \{[\bm{b}\times((\bm{v}\times\bm{b})\cdot\nabla\bm{b})]\times\bm{v}+(\bm{v}\cdot\bm{b})\bm{\kappa}\times\bm{v}\}\cdot\partial_{\bm{v}}\nonumber
    \\
    &+\frac{1}{4}\sin(2\theta)\{ (\bm{v}_\perp\cdot\nabla\bm{b})\times\bm{v}-[\bm{b}\times((\bm{v}\times\bm{b})\cdot\nabla\bm{b})]\times\bm{v}\}\cdot\partial_{\bm{v}}\nonumber
    \\
    &-\frac{1}{4}\cos(2\theta) \{[\bm{b}\times(\bm{v}_\perp\cdot\nabla\bm{b})]\times\bm{v}
    +[(\bm{v}\times\bm{b})\cdot\nabla\bm{b}]\times\bm{v}\}\cdot\partial_{\bm{v}}
    \nonumber
    \\
    &-(\cos\theta\bm{v}\times\bm{b}-\sin\theta\bm{b}\times(\bm{v}\times\bm{b}))\cdot\nabla\ln\omega_c \bm{v}\times\bm{b}\cdot\partial_{\bm{v}}
\end{align}
Since $\mathbf{d}\vartheta_0$ is independent of $\bm{v}$, we only need the $\bm{x}$-components of the vector fields $I_0\widetilde{X}^\theta_1$ and $\mathcal{L}_{R_0}I_0\widetilde{X}^\theta_1$, which finally provide the expression for the first nonvanishing term in the adiabatic invariant series
\begin{align}
    \mu_2=\frac{m}{2}\frac{|\bm{v}\times\bm{b}|^2}{\omega_c}.
\end{align}
As is straighforward to verify, the Hessian along the normal direction $(v^1,v^2)$ is sign-definite
\begin{align}
    \mathbf{H}_\perp(\mu_2)=
    \frac{m}{\omega_c}\begin{pmatrix}
    1 & 0 \\
    0 & 1
    \end{pmatrix},
\end{align}
confirming the normal stability of the slow-manifold via Theorem \ref{free_action_thm}.


\subsection{The proper-time relativistic Pauli embedding\label{rel_pauli_sec}}
As a second application, we generalize the discussion from Section \ref{eq:classical_pauli} to the Lorentz-covariant relativistic setting. We begin by recalling the standard Lorentz-covariant Hamiltonian formulation of charged particle dynamics on a flat Minkowski spacetime $(M,\langle \cdot,\cdot\rangle)$, whose inner product $\langle\cdot,\cdot\rangle$ has mostly-positive signature. We denote spacetime events using the symbol $R\in M$, and elements of the spacetime tangent bundle $TM\approx M\times M$ as $(R,V)\in T_RM$. The electromagnetic field is specified by a $1$-form $A$ on $M$, whose exterior derivative gives the Faraday $2$-form $F = \mathbf{d}A$. By way of the Minkowski inner product, the Faraday $2$-form induces a Faraday tensor $\bm{F}:TM\rightarrow TM$, defined so that $\langle V_1,\bm{F}\,V_2\rangle = \iota_{V_2}\iota_{V_1}F$ for each pair of vector fields $V_1,V_2$ on $M$. An individual charged particle with $4$-position $R$, $4$-velocity $V$, and proper time $\tau$ moves through such a spacetime according to the relativistic Newton-Lorentz equations,
\begin{align}
\frac{dV}{d\tau} & = \zeta\,\bm{F}(R)\,V,\quad \frac{dR}{d\tau}  = \epsilon\,V. \label{NL_accelerattion}
\end{align}
We have written \eqref{NL_accelerattion} in dimensionless form. To recover dimensional results, introduce the particle charge $q$, the particle mass $m$, the speed of light $c$, a spacetime length scale $L$, and a characteristic magnetic field strength $B_0$ (physically interpreted as the characteristic size of the Lorentz scalar $\sqrt{|\bm{B}|^2 - |\bm{E}|^2}$). The dimensional $4$-position, $4$-velocity, proper time, and Faraday tensor are then given by $LR$, $c\,V$, $(mc/q B_0)\tau$, and $B_0\,L^2\,F$, respectively. The correct physical interpretations of the constants $\zeta$ and $\epsilon$ are therefore $\zeta = q/|q|$ and
\begin{align}
    \epsilon = \frac{mc^2}{|q|\,B_0\,\,L},\label{relativistic_eps_def}
\end{align}
where the latter represents the ratio of the so-called ``light radius" $\rho_c = (mc^2)/(|q|B_0)$ to the characteristic field scale length $L$, $\epsilon = \rho_c/L$.

Going forward, we will assume that the electromagnetic potential $A$ decomposes as the sum $A =A_0 + \epsilon\,A_1$, where the ``$\bm{E}\cdot\bm{B}$'' and ``$|\bm{B}|^2 - |\bm{E}|^2$'' Lorentz scalars associated with $F_0 = \mathbf{d}A_0$ are zero and positive, respectively. This is a Lorentz-covariant way of asserting the system is strongly magnetized. As such, we refer to this assumption as the \textbf{magnetization assumption}. 

Under the magnetization assumption, we claim the Newton-Lorentz equations \eqref{NL_accelerattion} comprise a nearly-periodic Hamiltonian system $X_\epsilon = (dR/d\tau,dV/d\tau)$ on the exact, regular, barely-symplectic manifold $(TM,\Omega_\epsilon)$. We argue as follows. The barely-symplectic form is given by $\Omega_\epsilon = -\mathbf{d}\vartheta_\epsilon$, where
\begin{align}
    \vartheta_\epsilon = \zeta\,(A_0+\epsilon\,A_1) + \epsilon\,\langle V,dR\rangle. \label{relativistic_one_form}
\end{align}
The degeneracy index for $\Omega_\epsilon$ is $d=2$, as in the non-relativistic case. The system Hamiltonian is $H_\epsilon(R,V)  = \tfrac{1}{2}\epsilon^2\,\langle V,V\rangle.$
It is straightforward to confirm that the Hamilton equation $\iota_{X_\epsilon}\Omega_\epsilon = \mathbf{d}H_\epsilon$ recovers \eqref{NL_accelerattion} for all values of $\epsilon$, which confirms the Hamiltonian nature of the relativistic Newton-Lorentz equations. To demonstrate that $X_\epsilon$ is nearly-periodic, we must show $X_0 = \omega_0\,R_0$, where $\omega_0$ is some nowhere-vanishing smooth function and $R_0$ is the generator of a $U(1)$-action on $TM$ that satisfies $\mathcal{L}_{R_0}\omega_0 = 0$. For this, we turn to the following Lemma.

\begin{lemma}\label{NL_nearly_periodic_structure}
Under the magnetization assumption, the smooth function $\omega_0 = \sqrt{-\mathrm{tr}(\bm{F}_0^2)/2}$ is nowhere-vanishing. In addition, the vector field 
\begin{align}
    R_0 = \frac{\zeta}{\omega_0}\,\bm{F}_0\,V\,\partial_V,
\end{align}
is the generator of a $U(1)$-action on $TM$, and $\mathcal{L}_{R_0}\omega_0 = 0$.
\end{lemma}
\begin{proof}
First we note that $\text{tr}(\bm{F}_0^2) = 2(|\bm{E}_0|^2-|\bm{B}_0|^2)$.  According to the magnetization assumption, we therefore have $-\text{tr}(\bm{F}_0^2)>0$ on $M$. It immediately follows that $\omega_0$ is real-valued and nowhere-vanishing, as required. We also note that $\mathcal{L}_{R_0}\omega_0 = 0$ is obvious since $\omega_0$ depends only on $R$, while $R_0$ has no $R$-component.

Next we identify the dimension of the null space for $\bm{F}_0$. With respect to an orthonormal basis $(q_0,q_1,q_2,q_3)$ for $T_RM$ such that $q_0$ is timelike, the coefficient matrix for $\bm{F}_0(R)$, $[\bm{F}_0]$, is given by $[\bm{F}_0] = [g][F]$, where $[g]_{ij} = \langle q_i,q_j\rangle$ is diagonal symmetric and $[F]_{ij} = F(e_i,e_j)$ is antisymmetric. Because the $\bm{E}\cdot\bm{B}$ Lorentz scalar vanishes for $\bm{F}_0$, we must have $ 0= \text{det}\,[\bm{F}_0]=(\text{det}\,[g])(\text{det}\,[F]) = -\text{det}\,[F]$. In other words, the antisymmetric matrix $[F]$ must have a non-trivial null space. And since, by hypothesis, $[F]$ does not vanish, the block normal form for antisymmetric matrices implies that the spectrum for $[F]$ must be of the form $(i\lambda,-i\lambda,0,0)$, where $\lambda>0$. In particular, the null space $K_R\subset T_RM$ for $\bm{F}_0(R)$ must be $2$-dimensional.

Now we will characterize the behavior of $\bm{F}_0$ on the subspace orthogonal to its null space. Let $K_R^\perp\subset T_RM$ be the orthogonal complement to the null space $K_R$. If $V\in K_R^\perp$ and $W\in K_R$ then $\langle W,\bm{F}_0(R)\,V\rangle = -\langle V,\bm{F}_0(R)\,W\rangle = 0$. Therefore $K_R^\perp$ is a $2$-dimensional invariant subspace for $\bm{F}_0(R)$ complementary to $K_R$. Let $\bm{F}_0^\perp(R):K_R^\perp\rightarrow K_R^\perp$ denote the restriction of $\bm{F}_0(R)$ to $K_R^\perp$. By the Cayley-Hamilton theorem for $2\times 2$ matrices, we have
\begin{align}
    (\bm{F}_0^\perp(R))^2  +\text{det}(\bm{F}_0^\perp(R))\,\mathbb{I}^{\perp}=0,\label{ch_thm}
\end{align}
where we have used $\text{tr}(\bm{F}_0^\perp(R)) = \text{tr}(\bm{F}_0(R)) = 0$ and introduced the identity map $\mathbb{I}^\perp:K_R^\perp\rightarrow K_R^\perp$. Taking the trace of \eqref{ch_thm}, we also obtain
\begin{align}
    \text{tr}((\bm{F}_0^{\perp}(R))^2)  + 2\,\text{det}(\bm{F}_0^\perp(R)) = 0.\label{tr_ch_thm}
\end{align}
Combining \eqref{ch_thm} and \eqref{tr_ch_thm}, we find,
\begin{align}
   \left( \frac{\bm{F}_0^\perp(R) }{\sqrt{-\text{tr}((\bm{F}_0^{\perp}(R))^2)/2}}\right)^2=  -\mathbb{I}^{\perp}\label{acs}
\end{align}
where we have used $\text{tr}((\bm{F}_0^{\perp})^2) = 2(E^2-B^2)<0$ to ensure the square root is real. This identity says $\bm{F}_0^\perp(R)/\omega_0(R)$ is a complex structure on the vector space $K_R^\perp$ for each $R$.

Finally we determine the integral curves of the vector field $R_0$. If $(R(\lambda),V(\lambda))$ is such an integral curve, then the component curves satisfy the system of ordinary differential equations
\begin{align*}
    \frac{dV}{d\lambda}  &  = \frac{\zeta}{\omega_0}\,\bm{F}_0\,V,\quad \frac{dR}{d\lambda}  = 0.
\end{align*}
Clearly $R(\lambda) = R(0)$. For the $4$-velocity, we write $V(\lambda) = V_{\parallel}(\lambda) + V_{\perp}(\lambda)$, where $V_{\parallel}$ denotes the orthogonal projection into the null space $K_{R(0)}$ and $V_\perp$ denotes the orthogonal projection into $K_R^\perp$. These projected curves satisfy the linear system
\begin{align*}
    \frac{dV_\parallel}{d\lambda} & = 0,\quad \frac{dV_\perp}{d\lambda}  = \frac{\zeta}{\omega_0}\,\bm{F}_0^\perp\,V_\perp.
\end{align*}
We obviously have $V_\parallel(\lambda) = V_{\parallel}(0)$. To solve the $V_\perp$ equation, we observe that \eqref{acs} implies $[\frac{\zeta}{\omega_0}\,\bm{F}_0^\perp]^2 = -\mathbb{I}_\perp$, and then recognize that we can compute the matrix exponential $\exp(\lambda \frac{\zeta}{\omega_0}\,\bm{F}_0^\perp)$ exactly using Euler's formula $\exp(i\theta) = \cos\theta + i\,\sin\theta$, giving
\begin{align}
    V_\perp(\lambda) &=\exp(\lambda\,\frac{\zeta}{\omega_0}\,\bm{F}_0^\perp)\,V_\perp(0)\nonumber\\
    & = \bigg(\cos\lambda\,\mathbb{I}_\perp + \sin\lambda\,\frac{\zeta}{\omega_0}\,\bm{F}_0^\perp\bigg)\,V_\perp(0).\nonumber
\end{align}
By $2\pi$-periodicity of the solutions thus obtained, we conclude that $R_0$ generates a $U(1)$-action given explicitly by
\begin{align}
    \Phi_\theta(R,V) = (R,P_\parallel\,V + [\cos\theta\,\mathbb{I}_\perp + \sin\theta\,\zeta\,\bm{F}_0^\perp/\omega_0]P_\perp\,V),\label{relativistic_circle_action}
\end{align}
where $P_\perp$ and $P_\parallel$ denote orthogonal projections into $K_R^\perp$ and $K_R$, respectively.
\end{proof}

Note that in the process of proving the above Lemma we identified important structural properties of $\bm{F}_0$. These are summarized in the following definition.

\begin{definition}
The \textbf{parallel flat} is the subbundle $K\subset TM$ whose fiber at $R\in M$ is the $2$-dimensional null space of $\bm{F}_0(R)$. The \text{perpendicular flat} is the orthogonal complement bundle $K^\perp$. The \textbf{orthogonal projections} into $K$ and $K^\perp$ are given by $P_\parallel:TM\rightarrow TM$ and $P_\perp:TM\rightarrow TM$, respectively, where
\begin{align}
    P_\perp = -\frac{\bm{F}_0^2}{\omega_0^2},\quad P_\parallel = \mathbb{I} -P_\perp.\label{projector_formulas}
\end{align}
\end{definition}



We may now define the equations of motion for a relativistic Pauli particle, and study their properties. The \textbf{relativistic Pauli Hamiltonian} $\mathcal{H}_\epsilon:TM\rightarrow \mathbb{R}$ is given by
\begin{align}
    \mathcal{H}_\epsilon(R,V) & = \frac{1}{2}\epsilon^2\,\langle V,V\rangle +\epsilon^2\,\mathcal{M}\,\omega_0,
\end{align}
where $\mathcal{M}\in\mathbb{R}$ is a parameter and $\omega_0$ is defined in Lemma \ref{NL_nearly_periodic_structure}. The \textbf{relativistic Pauli system} is the vector field $\mathcal{X}_\epsilon$ defined by the Hamilton equation $\iota_{\mathcal{X}_\epsilon}\Omega_\epsilon = \mathbf{d}\mathcal{H}_\epsilon$ , where $\Omega_\epsilon = -\mathbf{d}\vartheta_\epsilon$ with $\vartheta_\epsilon$ given in \eqref{relativistic_one_form}. As in the non-relativistic case, in defining this Pauli system we have left the Lorentz symplectic structure unchanged while adding a Pauli potential $\mathcal{M}\,\omega_0$ to the Lorentz Hamiltonian. Also is parallel with the non-relativistic case, the relativistic Pauli system admits a non-degenerate Lagrangian structure with Lagrangian $L(R,\dot{R}) = \epsilon\,\tfrac{1}{2}\langle \dot{R},\dot{R}\rangle + \zeta\,\iota_{\dot{R}}A - \epsilon\,\mathcal{M}\,\omega_0(R)$. The following analysis will demonstrate that the relativistic guiding center equations, as derived originally by Boghosian \cite{Boghosian_2003}, are embedded within the relativistic Pauli system as a slow manifold, and that this slow manifold enjoys long-term normal stability. In so doing we will generalize the observations of Xiao and Qin \cite{Xiao_2020} to allow for time-dependent electromagnetic fields, strong $E\times B$ drifts, and all special relativistic effects such as time dialation. In addition, we will generalize our result on continuous-time normal stability of the Pauli embedding to the relativistic setting.

First we observe that $\mathcal{X}_\epsilon$ is a nearly-periodic system. To show this, we note that $\mathcal{X}_\epsilon$ is given explicitly by $\mathcal{X}_\epsilon = (dR/d\tau,dV/d\tau)$ with
\begin{align}
    \frac{dR}{d\tau} & = \epsilon\,V,\quad \frac{dV}{d\tau} = \zeta\,(\bm{F}_0+\epsilon\,\bm{F}_1)\,V - \epsilon\,\mathcal{M}\,\nabla \omega_0.\label{relativistic_pauli_system}
\end{align}
Here $\nabla$ denotes the gradient operator associated with the Minkowski inner product. These equations differ from the Newton-Lorentz equations by a single $O(\epsilon)$ term. Therefore $\mathcal{X}_0 = \omega_0\,R_0$, where $\omega_0$ and $R_0$ are defined as they were for the Newton-Lorentz system. Lemma \ref{NL_nearly_periodic_structure} therefore implies $\mathcal{X}_\epsilon$ is nearly-periodic, as claimed.

Next we will show that the relativistic Pauli system is fast-slow in order to efficient identify the system's slow manifold and the corresponding induced slow dynamics.

\begin{lemma}\label{frame_prop}
There exists a smooth orthonormal tetrad $(e_0,e_1,e_2,e_3)$ on $M$ such that $(e_0,e_3)$ frames the null-space bundle $K\subset TM$ and $(e_1,e_2)$ frames $K^\perp\subset TM$. Moreover, $e_0$ is timelike and $e_k$ is spacelike for $k=1,2,3$.
\end{lemma}

\begin{lemma}
In the coordinates $(R,V^0,V^1,V^2,V^3)$ on $TM$ defined by 
\begin{align}
V = V^0\,e_0 + V^1 \,e_1 + V^2\,e_2 + V^3\,e_3,\nonumber
\end{align}
where $(e_0,e_1,e_2,e_3)$ is the orthonormal tetrad provided by Proposition \ref{frame_prop}, the relativistic Pauli system \eqref{relativistic_pauli_system} is equivalent to 
\begin{align}
\frac{dV^0}{d\tau} &=\epsilon\, \langle  \zeta\,\bm{{F}}_1\,e_0,V\rangle  -\epsilon\, \left\langle V_\perp,\nabla_V\left[\frac{\bm{{F}}_0^2}{\omega_0^2}\right]e_0 \right\rangle - \epsilon\,V^3\mathcal{Q}(V)+\epsilon\,\mathcal{M}\,\langle  e_0,\nabla\omega_0\rangle\label{v0dot_prop}\\
\frac{dV^1}{d\tau} &= -\langle \zeta\,(\bm{{F}}_0+\epsilon\,\bm{F}_1)\,e_1,V\rangle   -\epsilon\,\left\langle V_\parallel,\nabla_V\left[\frac{\bm{{F}}_0^2}{\omega_0^2}\right]\,e_1  \right\rangle + \epsilon\,V^2\mathcal{R}(V)-\epsilon\,\mathcal{M}\,\langle e_1,\nabla\omega_0 \rangle\\
\frac{dV^2}{d\tau} &=- \langle \zeta\,(\bm{{F}}_0+\epsilon\,\bm{F}_1)\,e_2,V\rangle  -\epsilon\,\left\langle V_\parallel,\nabla_V\left[\frac{\bm{{F}}_0^2}{\omega_0^2}\right]\,e_2  \right\rangle -\epsilon \,V^1\mathcal{R}(V)-\epsilon\,\mathcal{M}\,\langle e_2,\nabla\omega_0 \rangle\\
\frac{dV^3}{d\tau} & =  - \epsilon\,\langle \zeta\, \bm{{F}}_1\,e_3,V\rangle +\epsilon\,\left\langle V_\perp, \nabla_V\left[\frac{\bm{{F}}_0^2}{\omega_0^2}\right]\,e_3 \right\rangle - \epsilon\,V^0\mathcal{Q}(V) - \epsilon\,\mathcal{M}\,\langle e_3,\nabla\omega_0\rangle.\\
\frac{dR}{d\tau} & = \epsilon\,V.\label{rdot_prop}
\end{align}
Here the $1$-forms $\mathcal{Q}$ and $\mathcal{R}$ are defined according to $\mathcal{Q}(V) = \langle \nabla_Ve_0,e_3\rangle $ and $\mathcal{R}(V) = \langle \nabla_Ve_1,e_2\rangle$, and we have introduced the shorthand $V_\parallel = P_\parallel V\in K_R$, $V_\perp = P_\perp V\in K_R^\perp$.

\end{lemma}

\begin{proposition}
The system of ordinary differential equations \eqref{v0dot_prop}-\eqref{rdot_prop} comprises a fast-slow system with slow variable $x = (R,V^0,V^3)$ and fast variable $y = (V^1,V^2)$. The function $f_\epsilon(x,y)=(dV^1/d\tau,dV^2/d\tau)$ is given by
\begin{align}
f_0(x,y) & = \begin{pmatrix}
-\zeta\,V^2\langle \bm{{F}}_0\,e_1,e_2\rangle\\
\zeta\,V^1\langle \bm{{F}}_0\,e_1,e_2\rangle
\end{pmatrix},
\end{align}
\begin{align}
f_1(x,y) & = \begin{pmatrix}
- \langle \zeta\,\bm{{F}}_1\,e_1,V\rangle -\left\langle V_\parallel,\nabla_V\left[\frac{\bm{{F}}_0^2}{\omega_0^2}\right]\,e_1  \right\rangle + V^2\mathcal{R}(V) - \mathcal{M} \langle e_1, \nabla\omega_0\rangle\\
- \langle \zeta\,\bm{{F}}_1\,e_2,V\rangle -\left\langle V_\parallel,\nabla_V\left[\frac{\bm{{F}}_0^2}{\omega_0^2}\right]\,e_2  \right\rangle - V^1\mathcal{R}(V) -\mathcal{M}\langle e_2,\nabla\omega_0\rangle
\end{pmatrix}.
\end{align}
\end{proposition}

\begin{proposition}
The first two coefficients of the formal slow manifold $y_\epsilon^*= (V^{1*}_\epsilon,V^{2*}_\epsilon)$ for the fast-slow system \eqref{v0dot_prop}-\eqref{rdot_prop} are given by
\begin{align}
\begin{pmatrix}
V^{1*}_0\\
V^{2*}_0
\end{pmatrix}=
\begin{pmatrix}
0\\
0
\end{pmatrix},
\end{align}
\begin{align}
\begin{pmatrix}
V^{1*}_1\\
V^{2*}_1
\end{pmatrix}=
\frac{\zeta\,\langle e_2,\bm{F}_0\,e_1\rangle }{\omega_0^2}\begin{pmatrix}
\langle \zeta\,\bm{F}_1\,e_2,V_\parallel\rangle + \langle V_\parallel,\nabla_{V_\parallel}\left[\frac{\bm{F}_0^2}{\omega_0^2}\right]\,e_2\rangle + \mathcal{M}\langle e_2,\nabla\omega_0\rangle\\
-\langle \zeta\,\bm{F}_1\,e_1,V_\parallel\rangle - \langle V_\parallel,\nabla_{V_\parallel}\left[\frac{\bm{F}_0^2}{\omega_0^2}\right]\,e_1\rangle - \mathcal{M}\langle e_1,\nabla\omega_0\rangle
\end{pmatrix}.
\end{align}
In particular, if $(V_\perp)^*_\epsilon = (V^1)^*_\epsilon\,e_1 + (V^2)^*_\epsilon\,e_2$, we have
\begin{align}
    (V_\perp)^*_\epsilon = \epsilon\,\frac{\zeta\,\bm{F}_0}{\omega_0^2}\bigg( \zeta\,\bm{F}_1\,V_\parallel + \nabla_{V_\parallel}\left[\frac{\bm{F}_0^2}{\omega_0^2}\right]\,V_\parallel - \mathcal{M}\nabla\omega_0\bigg)+O(\epsilon^2).
\end{align}
\end{proposition}

Now it is simple to demonstrate that slow manifold dynamics relativistic Pauli system approximately agrees with the covariant guiding center theory developed by Boghosian \cite{Boghosian_2003}. As explained in \cite{Burby_loops_2019}, dynamics on the slow manifold is necessarily Hamiltonian. The corresponding symplectic form $\Omega^* = -\mathbf{d}\vartheta_\epsilon^*$ is given by pulling back $\Omega_\epsilon$ along the mapping $x\mapsto (x,y_\epsilon^*)$, which leads to
\begin{align}
    \vartheta_\epsilon^* = \zeta\,A\,+\epsilon\,\langle V_\parallel,dR\rangle + O(\epsilon^2).
\end{align}
This $1$-form agrees with the $1$-form reported in Eq.\,(3.489) in \cite{Boghosian_2003} to the displayed order. For agreement, we use $V_\parallel = V^1\,e_1 + V^2\,e_2 = K\hat{\bm{t}}$, where $K$ and $\hat{\bm{t}}$ are defined by Boghosian. The slow manifold Hamiltonian is given by pulling back the Pauli Hamiltonian along the same map, giving
\begin{align}
   \mathcal{H}_\epsilon^* = \epsilon^2\,\left(\frac{1}{2}\langle V_\parallel,V_\parallel\rangle + \mu\,\omega_0 \right) + O(\epsilon^3).
\end{align}
Since $\langle V_\parallel,V_\parallel\rangle = -K$, where the right-hand-side uses Boghosian's notation, this Hamiltonian agrees with Eq.\,(3.488) from \cite{Boghosian_2003} to the displayed order. We conclude that slow manifold dynamics for the relativistic Pauli system agree with relativistic guiding center theory to the same order as in the non-relativistic case. This implies in particular that the strategy underlying Xiao and Qin's numerical integration scheme \cite{Xiao_2020} may be applied in the covariant relativistic setting as well.

Finally, we will demonstrate that the slow manifold for the relativistic Pauli system enjoys long-term normal stability. For this, we will show that the first nontrivial term in the adiabatic invariant for this system has sign-definite second variation along the limiting slow manifold $V^1 = V^2 = 0$. Let $\mu_\epsilon = \mu_0 + \epsilon\,\mu_1 + \epsilon^2\,\mu_2 + \dots$ denote the adiabatic invariant series for the Pauli system. According to Eq.\,(3.14) in \cite{Burby_Squire_2020}, $\mu_0 = \iota_{R_0}\langle \vartheta_0\rangle$, where $R_0$ is defined in Lemma \ref{NL_nearly_periodic_structure}, $\vartheta_0 = \zeta\,A_0$, and the angle brackets denote averaging over the $U(1)$-action $\Phi_\theta$ generated by $R_0$, i.e. that given in \eqref{relativistic_circle_action}. Since $\Phi_\theta$ leaves the $4$-position $R$ fixed $\langle \vartheta_0\rangle = \vartheta_0$, and since $R_0$ has only velocity components, we conclude $\mu_0=0$. According to Eq.\,(3.15) in \cite{Burby_Squire_2020}, $\mu_1 = \iota_{R_0}\langle \vartheta_1\rangle$, where $\vartheta_1 = \zeta\,A_1 + \langle V,dR \rangle$. Again using the fact that $\Phi_\theta$ leaves $R$ fixed, the average is simple to compute, giving $\langle \vartheta_1\rangle = \zeta\,A_1 + \langle V_\parallel,dR\rangle $. Therefore the contraction $\iota_{R_0}\langle \vartheta_1\rangle $ vanishes again, giving $\mu_1 =0$. Finally using Eq.\,(3.16) in \cite{Burby_Squire_2020}, we find that $\mu_2 = \frac{1}{2}\langle \mathbf{d}\vartheta_0(\mathcal{L}_{R_0}I_0\widetilde{\mathcal{X}}_1,I_0\widetilde{\mathcal{X}}_1)\rangle$, where $\mathcal{X}_1$ is the first-order term in $\mathcal{X}_\epsilon$, $\widetilde{\mathcal{X}}_1 =\mathcal{X}_1 - \langle \mathcal{X}_1\rangle $, and $I_0 = \mathcal{L}_{\omega_0\,R_0}^{-1}$. Since $\mathbf{d}\vartheta_0 = F_0$, we only need to compute the $R$-components of the vector fields $I_0\widetilde{\mathcal{X}}_1$ and $\mathcal{L}_{R_0}I_0\widetilde{\mathcal{X}}_1$. For this purpose, we observe that the $R$-component of the vector field $\mathcal{X}_1^\theta = \Phi_\theta^*\mathcal{X}_1$ is given by
\begin{align}
    (\mathcal{X}_1^\theta)^R = P_\parallel\,V + [\cos\theta\,\mathbb{I}_\perp + \sin\theta\,\zeta\,\bm{F}_0^\perp/\omega_0]P_\perp\,V,
\end{align}
from which we infer
\begin{align*}
    (I_0\widetilde{\mathcal{X}}_1^\theta)^R &= \frac{1}{\omega_0} [\sin\theta\,\mathbb{I}_\perp - \cos\theta\,\zeta\,\bm{F}_0^\perp/\omega_0]P_\perp\,V\\
    (\mathcal{L}_{R_0}I_0\widetilde{\mathcal{X}}_1^\theta)^R& = \frac{1}{\omega_0} [\cos\theta\,\mathbb{I}_\perp + \sin\theta\,\zeta\,\bm{F}_0^\perp/\omega_0]P_\perp\,V.
\end{align*}
The second-order adiabatic invariant is therefore
\begin{align}
\mu_2 & =\frac{1}{2} \frac{1}{2\pi}\int_0^{2\pi} \left\langle \frac{1}{\omega_0} [\cos\theta\,\mathbb{I}_\perp + \sin\theta\,\zeta\,\bm{F}_0^\perp/\omega_0]P_\perp\,V,\bm{F}_0\left(\frac{1}{\omega_0} [\sin\theta\,\mathbb{I}_\perp - \cos\theta\,\zeta\,\bm{F}_0^\perp/\omega_0]P_\perp\,V\right)\right\rangle\,d\theta\nonumber\\
& = -\frac{\zeta}{\omega_0}\frac{1}{2\pi}\int_0^{2\pi}\left\langle V_\perp,[\bm{F}_0/\omega_0]^2 V_\perp \right\rangle\,\sin^2\theta\,d\theta\nonumber\\
& = \frac{\zeta\,\langle V_\perp,V_\perp\rangle}{2\omega_0}.
\end{align}
Since $K_R^\perp$ is space-like for each $R$, the Hessian of $\mu_2$ along $V^1=V^2 = 0$ is sign-semi-definite, much as in the relativistic case. By theorem \ref{free_action_thm} with $d = \nu = 2$ and energy conservation, we conclude that if a trajectory for the relativistic Pauli system with a bounded spatial component $R(t)$ begins within $\epsilon$ of $V^1 = V^2 = 0$, then it will remain within $\epsilon^{1/2}$ over large time intervals. It is also not difficult to show using Theorem \ref{free_action_thm} that the normal deviation from an $N^{\text{th}}$-order slow manifold will be bound by $\epsilon^{(N+1)/2}$ for trajectories that begin within $\epsilon^{N+1}$.

\subsection{The symplectic Lorentz embedding\label{symplectic_lorentz_sec}}
As a final application, we will study a general method for embedding symplectic Hamiltonian systems as normally-stable elliptic slow manifolds in higher-dimension Lagrangian systems with regular Lagrangians. This method applies in particular to the non-canonical guiding center system, but differs from the Pauli embeddings studied in Sections \ref{eq:classical_pauli} and \ref{rel_pauli_sec} in an essential manner; where the dimensionality of either the non-relativistic or relativistic Pauli systems is two greater than that of the corresponding guiding center system, the dimension of the embedding space studied in this section is twice that of the system being embedded. We emphasize, however, that this method of embedding is \emph{not} equivalent to the method of formal Lagrangians \cite{Ibragimov_2006,Ibragimov_2007}. The formal Lagrangian technique does not embed using slow manifolds, and does not lead to regular Lagrangians, in contrast to the method described here. 

The basic idea behind our construction may be described as follows. Let $X$ be a Hamiltonian system with Hamiltonian $H$ on an exact symplectic manifold $(M,\beta)$ equipped with a Riemannian metric $g$ and symplectic form $\beta = -\mathbf{d}\alpha$. We would like to embed $X$ in a larger system with a regular Lagrangian structure. To do this, we consider the dynamics of a charged particle with small mass $\epsilon >0$ and unit positive charge moving on $M$. The magnetic field this particle experiences is given by the symplectic $2$-form $\beta$. The electric field is given by $-\nabla H$, i.e. the Hamiltonian serves as an electrostatic potential. As the particle mass $\epsilon$ becomes smaller, the timescale for gyration around the magnetic field shrinks. In fact, as the following detailed analysis will demonstrate, the metric $g$ on $M$ can always be chosen to ensure the particle's gyration around the magnetic field becomes periodic with short period as $\epsilon\rightarrow 0$. Therefore the equations of motion for this particle comprise a nearly-periodic system, and, according to the theory developed in this Article, admit slow manifolds of each order near which the rapid gyrations are suppressed. Strikingly, the corresponding slow manifold dynamics recover the original dynamics defined by $X$ to leading order in $\epsilon$. From the perspective of the small-mass particle moving on $M$, the original dynamics is recovered as a generalized $E\times B$-drift. Moreover, normal stability of the slow manifold emerges, by way of Theorem \ref{free_action_thm}, as a consequence of adiabatic invariance of a generalized magnetic moment.

As a first step in a detailed description of this embedding technique, we review the construction of a Riemannian metric ``compatible" with a given symplectic form $\beta$ on a manifold $M$. 

\begin{lemma}\label{compatible_lemma}
Given a symplectic manifold $(M,\beta)$, there exists a Riemannian metric $g$ on $M$ and an almost complex structure $\mathbb{J}:TM\rightarrow TM$ such that $g(V,W) = \beta(V,\mathbb{J}W)$ for each pair of vector fields $V,W$ on $M$.
\end{lemma}
\begin{proof}
The proof is well-known, see for instance \cite{Silva_book_2008}, but we give a reproduction here for completeness and to emphasize its constructive character.

Let $G$ be an arbitrary Riemannian metric on $M$. The symplectic form $\beta$ induces an antisymmetric, non-singular bundle map $\bm{\beta}_G:TM\rightarrow TM$ by requiring $G(V,\bm{\beta}_G\,W) = \beta(V,W)$ for each pair of vector fields $V,W$. The associated bundle map $\bm{S}_G = -\bm{\beta}_G\,\bm{\beta}_G$ is therefore symmetric positive-definite with symmetric positive-definite square root $\sqrt{\bm{S}}_G$. We define the desired almost complex structure $\mathbb{J}$ according to $\mathbb{J} = \bm{\beta}_G^{-1}\,\sqrt{\bm{S}}_G$. To check that this formula does indeed define an almost complex structre, we compute as follows:
\begin{align*}
    \mathbb{J}^2 & = \bm{\beta}_G^{-1}\,\sqrt{\bm{S}}_G\,\bm{\beta}_G^{-1}\,\sqrt{\bm{S}}_G\nonumber\\
    & = -\bm{\beta}_G^{-1}\,\bm{\beta}_G^{-1}\,\bm{\beta}_G\,\bm{\beta}_G\nonumber\\
    & = -\mathbb{I}.
\end{align*}
Here we have used the fact that $\bm{S}_G$, and therefore $\sqrt{\bm{S}}_G$, commutes with $\bm{\beta}_G$. We then define the Riemannian metric $g$ by requiring the desired identity $g(V,W) = \beta(V,\mathbb{J}\,W)$ holds for arbitrary vector fields $V,W$. To show that the tensor $g$, thus defined, is symmetric and positive definite, we observe
\begin{align*}
    g(V,W) & = \beta(V,\mathbb{J}W) = G(V,\bm{\beta}_G\,\bm{\beta}_G^{-1}\,\sqrt{\bm{S}}_G\,W) =G(V,\sqrt{\bm{S}}_G\,W) = g(W,V),
\end{align*}
by symmetry of $\sqrt{\bm{S}}_G$, and for non-zero $V$
\begin{align*}
    g(V,V) = G(V,\sqrt{\bm{S}}_G\,V)>0,
\end{align*}
by positive-definiteness of $\sqrt{\bm{S}}_G$.
\end{proof}

We will now describe our embedding technique in detail. Let $(M,\beta)$ be an exact symplectic manifold. Using Lemma \ref{compatible_lemma}, choose a Riemannian metric $g$ on $M$ and an almost-complex structure $\mathbb{J}$ such that $g(V,W) = \beta(V,\mathbb{J}W)$ for arbitrary vector fields $V,W$. Let $X$ be a Hamiltonian system on $M$ with Hamiltonian $H$. We would like to embed the dynamics defined by $X$ as slow manifold dynamics in a larger system with a regular Lagrangian structure.

The phase space for this larger system will be the tangent bundle $TM$, points of which will be denoted $(R,V)\in T_RM$. For $\epsilon\in\mathbb{R}$, we introduce the exact, regular, barely-symplectic form $\Omega_\epsilon = -\mathbf{d}\vartheta_\epsilon$ on $TM$ with primitive
\begin{align}
    \vartheta_\epsilon = \pi^*\alpha + \epsilon\,g_R(V,dR),
\end{align}
where $\pi:TM\rightarrow M$ denotes the tangent bundle projection and $\alpha$ is a primitive for $\beta = -\mathbf{d}\alpha$.
We also introduce the $\epsilon$-dependent Hamilton function $\mathcal{H}_\epsilon(R,V) = \epsilon\,H(R)+\epsilon^2\tfrac{1}{2}\,g_R(V,V)$. The \textbf{symplectic Lorentz system} is the smooth $\epsilon$-dependent vector field $\mathcal{X}_\epsilon$ on $TM$ defined by the Hamilton equation $\iota_{\mathcal{X}_\epsilon}{\Omega}_\epsilon = \mathbf{d}\mathcal{H}_\epsilon$. Our goal is to prove that the symplectic Lorentz system contains a slow manifold whose slow dynamics agrees with those of $X$ to leading order in $\epsilon$. Moreover, we would like to establish normal stability of this slow manifold using Theorem \ref{free_action_thm}. As in the previous sections on Pauli embeddings, we will proceed by showing $\mathcal{X}_\epsilon$ is a nearly-periodic Hamiltonian system, identifying the associated slow manifold, and then showing that the adiabatic invariant has sign-semi-definite second variation along the slow manifold.

To see that $\mathcal{X}_\epsilon$ is nearly-periodic, suppose that $(R(t),V(t))$ is an $\mathcal{X}_\epsilon$-integral curve. It is not difficult to show that this curve must satisfy the system of evolution equations
\begin{align}
    \frac{DV}{dt} = \mathbb{J}\,V - \nabla H,\quad \frac{dR}{dt} =\epsilon\,V,\label{symplectic_Lorentz_system}
\end{align}
where $DV/dt$ denotes the covariant derivative of $V$ along the curve $R$. Also the analogy to the charged particle is now transparent: $\mathbb{J}V$ is the ``$\bm{v}\times\bm{B}$" term and $-\nabla H$ is the "electrostatic electric field".
To see how \eqref{symplectic_Lorentz_system} emerge, write the Lagrangian in component form $L=(\alpha_i+\epsilon V^jg_{ij})\dot{R}^i-\epsilon(H+\epsilon V^iV^jg_{ij}/2)$ and obtain the Euler-Lagrange equations $\dot{R}^i=\epsilon V^i$ and $g_{ij}\dot{V}^j+\epsilon V^jV^k\Gamma_{ijk}=-\beta_{ij}V^j-\partial_iH$, where $\Gamma_{ijk}$ is the Christoffel symbol of the first kind. The vector form is then recovered after identifying $d(V^j\bm{e}_j)/dt\cdot\bm{e}_i=g_{ij}\dot{V}^j+\epsilon V^jV^k\Gamma_{ijk}$, $\bm{e}^i\partial_iH=\nabla H$, and $-\bm{e}^i\Omega_{ij}V^j=\bm{e}^ig_{ik}\mathbb{J}^k_{\ j}V^ j=\mathbb{J}V$.

In particular, when $\epsilon = 0$, we must have $dR/dt = 0$ and $dV/dt = \mathbb{J}\,V - \nabla H$. The solution to this limit system is $(R(t),V(t)) = \Phi_{\theta}(R(0),V(0))$, where $\Phi_\theta:TM\rightarrow TM$ is given by
\begin{align}
    \Phi_\theta(R,V) = (R,-\mathbb{J}\nabla H + \exp(\theta\,\mathbb{J})[V + \mathbb{J}\nabla H]).
\end{align}
Note that since $\exp(\theta\mathbb{J}) = \cos\theta\, \mathbb{I}+ \sin\theta\,\mathbb{J}$, $\Phi_\theta$ defines a $U(1)$-action on $TM$. We may therefore infer that the symplectic Lorentz system is a nearly-periodic system with angular frequency $\omega_0 = 1$ and limiting roto-rate 
\begin{align}
    \label{eq:symplectic_lorentz_limiting_rotorate}
    R_0 = (\mathbb{J}\,V - \nabla H)\,\partial_V.
\end{align}

Next we will show that the symplectic Lorentz system is fast-slow in order to efficiently identify the system’s slow manifold and the corresponding induced slow dynamics.


\begin{lemma}
In the coordinates $(R^i,V^i)$ on $TM$ 
the symplectic Lorentz system \eqref{symplectic_Lorentz_system} is equivalent to \begin{align}
\label{eq:symplectic_lorentz_r}
\dot{R}^i&=\epsilon V^i \\
\label{eq:symplectic_lorentz_v}
g_{ij}\dot{V}^j+\epsilon V^jV^k\Gamma_{ijk}&=-\beta_{ij}V^j-\partial_iH,
\end{align}
where $\Gamma_{ijk}=\tfrac{1}{2}(\partial_kg_{ij}+\partial_jg_{ki}-\partial_ig_{jk})$ is the Christoffel symbol of the first kind.
\end{lemma}

\begin{proposition}
The system of ordinary differential equations \eqref{eq:symplectic_lorentz_r}-\eqref{eq:symplectic_lorentz_v} comprises a fast-slow system with slow variable $x = (R^i)$ and fast variable $y = (V^i)$. The function $f_\epsilon(x,y)=(dV^i/dt)$ is given by
\begin{align}
f^i_0(x,y) & = -g^{ij}\beta_{jk}V^k-g^{ij}\partial_jH,
\end{align}
\begin{align}
f^i_1(x,y) & =-g^{ij}\Gamma_{jk\ell}V^kV^\ell,
\end{align}
and the function $g_\epsilon(x,y)=(dR^i/dt)$ is given by
\begin{align}
g^i_0(x,y) & = V^i.
\end{align}
\end{proposition}

\begin{proposition}
The first two coefficients of the formal slow manifold $y_\epsilon^*= ((V^i)^{*}_\epsilon)$ for the fast-slow system \eqref{eq:classical_pauli_system_x}-\eqref{eq:classical_pauli_system_v2} are given by
\begin{align}
(V^i)_0^\ast\beta_{ij} = \partial_j H
\end{align}
\begin{align}
(V^j)^{\ast}_0\partial_j(V^i)^\ast_0=-g^{ij}\beta_{jk}(V^k)^\ast_1-g^{ij}\Gamma_{jk\ell}(V^k)^\ast_0(V^\ell)^\ast_0
\end{align}
In particular, to leading order the slow dynamics are given by $\dot{R}= -\mathbb{J}\nabla H$, which recovers the original Hamiltonian dynamics on $M$.
\end{proposition}

Finally, we will demonstrate that the slow manifold for the symplectic Lorentz system enjoys long-term normal stability. Again, letting
$\mu_\epsilon = \mu_0 + \epsilon\,\mu_1 + \epsilon^2\,\mu_2 + \dots$ denote the adiabatic invariant series, we have $\mu_0 = \iota_{R_0}\langle \vartheta_0\rangle$, where $R_0$ is given by \eqref{eq:symplectic_lorentz_limiting_rotorate}, $\vartheta_0 = \pi^\ast\alpha$, and the angle brackets denote averaging over the $U(1)$-action $\Phi_\theta$ that is generated by $R_0$. $R_0$ only has a component along $V$, $\Phi_\theta$ leaves the $R$-position fixed. Then, since $\vartheta_0$ depends only on $R$, we have that $\langle \vartheta_0\rangle = \vartheta_0$, and consequently $\mu_0=0$. Given that $\mu_0=0$ the next candidate becomes $\mu_1 = \iota_{R_0}\langle \vartheta_1\rangle$, where we have $\vartheta_1 = g_R(V,dR)$. Then taking
the average of the pullback $\Phi_\theta^\ast (g_R(V,dR))$ with respect to $\theta$ provides 
$\langle\vartheta_1\rangle=g(-\mathbb{J}\nabla H,dR)$ which, again, is independent of $V$, therefore providing 
$\mu_1=0$. We are finally left to compute $\mu_2 = \frac{1}{2}\langle 
\mathbf{d}\vartheta_0(\mathcal{L}_{R_0}I_0\widetilde{\mathcal{X}}_1,I_0\widetilde{\mathcal{X}}_1)\rangle$, 
where $\mathcal{X}_1=V^i\partial_{R^i}-g^{ij}\Gamma_{jkl}V^kV^\ell\partial_{V^i}$ is the first-order term in 
$\mathcal{X}_\epsilon$, i.e., $\widetilde{\mathcal{X}}_1 =\mathcal{X}_1 - \langle \mathcal{X}_1\rangle 
$, and $I_0 = \mathcal{L}_{\omega_0\,R_0}^{-1}$. Since $\vartheta_0$ has only $R$ component which depends only 
on $R$, we only need the $R$ components of the vector field 
$\widetilde{\mathcal{X}}_1^\theta=\Phi_\theta^\ast\widetilde{\mathcal{X}}_1$ which is given by 
\begin{align}
    (\widetilde{\mathcal{X}}_1^\theta)^{R}=-\mathbb{J}\nabla H+(\cos\theta\, \mathbb{I}+\sin\theta\, \mathbb{J})(V+\mathbb{J}\nabla H).
\end{align}
From this we infer
\begin{align}
    (I_0\widetilde{\mathcal{X}}_1^\theta)^{R}&=(\sin\theta\, \mathbb{I}-\cos\theta\, \mathbb{J})(V+\mathbb{J}\nabla H),\\
    (\mathcal{L}_{\omega_0R_0}I_0\widetilde{\mathcal{X}}_1^\theta)^{R}&=(\cos\theta\, \mathbb{I}+\sin\theta\, \mathbb{J})(V+\mathbb{J}\nabla H),
\end{align}
and the second-order adiabatic invariant is therefore
\begin{align}
    \mu_2&=-\frac{1}{2}\frac{1}{2\pi}\int_0^{2\pi}\beta((\cos\theta\, \mathbb{I}+\sin\theta\, \mathbb{J})(V+\mathbb{J}\nabla H),(\sin\theta\, \mathbb{I}-\cos\theta\, \mathbb{J})(V+\mathbb{J}\nabla H))d\theta\nonumber\\
    &=\frac{1}{2}\beta(V+\mathbb{J}\nabla H,\mathbb{J}(V+\mathbb{J}\nabla H))\nonumber\\
    &=\frac{1}{2}g(V+\mathbb{J}\nabla H,V+\mathbb{J}\nabla H)
\end{align}
The Hessian of $\mu$ along $V=-\mathbb{J}\nabla H$ is the metric $g$ and thus sign-definite. Via Theorem \ref{free_action_thm}, the slow-manifold is then stable.

\section{Discussion}

In this Article, we established a free-action stability principle for a large and interesting class of elliptic slow manifolds, namely those that arise in Hamiltonian nearly-periodic systems. We applied this general theory to establish continuous-time normal stability of the slow manifold embedding of guiding center dynamics introduced by Xiao and Qin in \cite{Xiao_2020}. Moreover, we extended the Xiao-Qin embedding and it stability to the Lorentz covariant relativistic setting. Finally, we introduced a general method for embedding any Hamiltonian system on a symplectic manifold as a normally-stable elliptic slow manifold in a larger system with a regular Lagrangian.

Our normal stability results are based on exploiting adiabatic invariants. This idea is not new. MacKay highlights the method in his review article. However, it is not clear in general when a given slow manifold should satisfy a free-action principle. It is therefore striking that a large and interesting class of slow manifolds satisfies the free-action principle ``automatically."

One might attempt to establish a free-action principle for any nearly-periodic Hamiltonian system, in particular for nearly-periodic systems on symplectic, presymplectic, Poisson, or even Dirac manifolds. However, our results are not so general. Instead, we have assumed that phase space is equipped with a closed $2$-form that is non-degenerate  except in the limit of infinite timescale separation. Such singular symplectic structures arise frequently in applications, especially in plasma physics, where disparate timescales abound. The problem of extending our results to more general Hamiltonian structure deserves further attention.

Our method of embedding any symplectic Hamiltonian system as a slow manifold in a regular Lagrangian system suggests interesting further developments. At first glance it suggests that the method proposed by Xiao and Qin for developing structure-preserving integrators for guiding center dynamics may be extended to any Hamiltonian system. However, in his thesis \cite{Ellison_thesis} and subsequent work \cite{Ellison_2018}, Ellison showed unwittingly that integrators derived in this manner will generally suffer from parasitic instabilities. Thus, the continuous-time normal stability established in Section \ref{symplectic_lorentz_sec} may be broken after discretizing time. In future work, we plan to investigate strategies for discretizing the symplectic Lorentz system that do not destroy normal stability of the slow manifold, and that preserve structural properties of the underlying continuous-time dynamics.

\section*{Acknowledgments}
{The work of J.W.B. was supported by} the Los Alamos National Laboratory
LDRD program under project number 20180756PRD4. Work of E.H. was supported by the Academy of Finland grant no. 315278. Any subjective views or opinions expressed herein do not necessarily represent the views of the Academy of Finland or Aalto University.

\bibliographystyle{unsrt}
\bibliography{cumulative_bib_file.bib}

\end{document}